\documentclass[12pt]{amsart}
\pdfoutput=1
\usepackage{amssymb,amscd}
\usepackage{graphicx}
\usepackage{mathrsfs}
\usepackage{dsfont}
\usepackage[left=1.3in,right=1.3in,bottom=1.3in,top=1.3in]{geometry}
\PassOptionsToPackage{hyphens}{url}
\usepackage[all,2cell]{xy}
\UseTwocells

\newcommand{\SchS}{(Sch/S)_{fppf}}
\newcommand{\ObSchS}{{\rm{Ob}}((Sch/S)_{fppf})}
\newcommand{\SchU}{(Sch/U)_{fppf}}
\newcommand{\SchV}{(Sch/V)_{fppf}}
\newcommand{\SchUU}{(Sch/U')_{fppf}}
\newcommand{\SchW}{(Sch/W)_{fppf}}
\newcommand{\ObSchU}{{\rm{Ob}}((Sch/U)_{fppf})}

\newcommand{\calx}{\mathcal X}
\newcommand{\caly}{\mathcal Y}
\newcommand{\cals}{\mathcal S}
\newcommand{\calz}{\mathcal Z}
\newcommand{\calc}{\mathcal C}
\newcommand{\Ob}{{\rm{Ob}}}
\newcommand{\N}{\mathbb N}

\newcommand{\cal}{\mathcal}

\renewcommand{\c}{\circ}

 \theoremstyle{plain}
\newtheorem{theorem}{Theorem}[section]
\newtheorem{corollary}[theorem]{Corollary}
\newtheorem{lemma}[theorem]{Lemma}
\newtheorem{proposition}[theorem]{Proposition}
\newtheorem{definition-proposition}[theorem]{Definition/Proposition}

\newtheorem{definition}[theorem]{Definition}

 \theoremstyle{definition}

\newtheorem{definition1}[theorem]{Definition}

\theoremstyle{remark}

\newtheorem{remark}[theorem]{Remark}

\numberwithin{equation}{section}

\makeatletter
\def\@seccntformat#1{\@ifundefined{#1@cntformat}%
    {\csname the#1\endcsname\quad}%      default
    {\csname #1@cntformat\endcsname}}%   individual control
\newcommand{\section@cntformat}{\S\thesection.\enspace}
\newcommand{\subsection@cntformat}{\S\thesubsection.\enspace}
\newcommand{\subsubsection@cntformat}{\S\thesubsubsection\enspace}
\makeatletter

\usepackage[dvipsnames, svgnames, x11names]{xcolor}
\definecolor{cite}{rgb}{0.50,0.00,1.00}
\definecolor{url}{rgb}{0.00,0.50,0.75}
\definecolor{link}{rgb}{0.00,0.00,0.50}
\usepackage[colorlinks,linkcolor=RoyalBlue,urlcolor=url,citecolor=magenta,breaklinks]{hyperref}

\usepackage{tikz-cd}
\usepackage{capt-of}
\usepackage{enumitem}

\begin{document}
\title{Perfect algebraic stacks}

\author{Tianwei Liang}

\begin{abstract}
We develop a theory of perfect algebraic stacks that extend our theory of perfect algebraic spaces in \cite{Liang}, \cite{Liang1} to the setting of algebraic stacks. We prove several desired properties of perfect algebraic stacks. This extends some previous results of perfect schemes and perfect algebraic spaces, including the recent one developed by Bertapelle et al. in \cite{Bertapellea}. Moreover, our theory extends the previous one developed by Xinwen Zhu in \cite{Zhu1}.

Our method to define perfect algebraic stacks differs from all previous approaches, as we utilize representability of algebraic spaces. There is a natural notion of algebraic Frobenius morphisms of algebraic stacks. The algebraic Frobenius morphism provides one with an explicit description of the perfection of an algebraic stack. This gives rise to the perfection functor on algebraic stacks, which enables us to pass between the usual and the perfect world.
\end{abstract}

\maketitle

\tableofcontents

\section{Introduction}
Let $p$ be a prime number and let $\mathbb{F}_{p}$ be a finite field of order $p$. All rings will be tacitly commutative with identity.
\subsection{Motivation}
The notion of perfect rings in commutative algebra is particularly important in algebraic geometry. Many significant researches in algebraic geometry are surrounding the setting of perfectness, see \cite{Scholze1}, \cite{Liu}, and \cite{Scholze2} for example. One naturally desires to generalize perfect rings to the setting of schemes. In the classical paper \cite{Serre}, Serre introduced the so-called perfect varieties. However, a perfect variety is in general not a scheme. Until another classical paper of Greenberg \cite{Greenberg}, the subject of perfect schemes comes into being. In \cite{Greenberg}, Greenberg introduced the notions of perfect closures of rings and schemes. This naturally gives rise to the so-called perfection functor. The perfection functor plays a significant role in many areas of algebraic geometry, see \cite{Liu}, \cite{Bertapelle1}, \cite{Boyarchenko}, and \cite{Bertapellea}.

For greater generality, one desires to generalize perfect schemes to the setting of algebraic spaces. Once we get perfect algebraic spaces, one would like to generalize it to perfect algebraic stacks. In \cite{Zhu} and \cite{Zhu1}, Xinwen Zhu formalized the notions of perfect algebraic spaces, and more generally, perfect algebraic stacks. However, Zhu's perfect algebraic spaces and perfect algebraic stacks are only defined over a perfect field of characteristic $p$. It would be desirable if one can extend perfect algebraic spaces and perfect algebraic stacks to be over some arbitrary base scheme. Another restriction is that Zhu's definitions depends on the Frobenius morphisms. For arbitrary algebraic space or algebraic stack $X$, the Frobenius morphism $X\rightarrow X$ may not make sense. One has no idea to decide the endomorphism of a set or a category to be Frobenius. In fact, we even do not know the characteristic of an algebraic space or an algebraic stack.

In \cite{Liang}, we define perfect algebraic spaces in a totally different approach, which turns out to solve all these problems. Next, in \cite{Liang1}, we construct the perfection of any algebraic space in characteristic $p$. This naturally gives rise to the perfection functor on algebraic spaces. The perfection functor enjoys several desirable properties. It enables us to pass between the usual and the perfect world.

There are also other references concerning perfect algebraic stacks, see \cite{Francis}, \cite{Hall}, and \cite{Lurie}. However, they are in the setting of derived algebraic geometry and their terminologies are completely different to us.

\subsection{Results}
This article is the subsequence of \cite{Liang} and \cite{Liang1}. In this article, we will continue our research by generalizing our perfect algebraic spaces to perfect algebraic stacks. Rather than utilizing any endomorphism $\calx\rightarrow\calx$ of an algebraic stack $\calx$ to make it perfect, we will make use of the representability by algebraic spaces of $1$-morphisms. This makes our approach different to all previous ones.

We begin by introducing the notions of perfect categories fibred in groupoids and perfect $1$-morphisms. The material forms the foundation of the sequent sections. Our approach gives rise to several types of perfect algebraic stacks. Let $S$ be some base scheme. Let $0\leq i\leq3$ be an integer. Let $AStack_{S}$ be the category of algebraic stacks over $S$.

Let $\textrm{Perf}AStack_{S}$, $\textrm{\underline{Perf}}AStack^{i}_{S}$, $\textrm{\underline{QPerf}}AStack^{i}_{S}$, $\textrm{\underline{SPerf}}AStack^{i}_{S}$, and $\textrm{\underline{StPerf}}AStack^{i}_{S}$ denote the 2-categories of perfect algebraic stacks respectively. Let $\textrm{Perf}DM_{S}$, $\textrm{\underline{Perf}}DM_{S}$, $\textrm{\underline{QPerf}}DM^{i}_{S}$, $\textrm{\underline{SPerf}}DM^{i}_{S}$, and $\textrm{\underline{StPerf}}DM^{i}_{S}$ denote the 2-categories of perfect Deligne-Mumford stacks respectively. These 2-categories all enjoy the following desirable property, which generalizes our result in \cite[Proposition 1.1]{Liang}.
\begin{proposition}
The $2$-categories ${\rm{Perf}}AStack_{S}$, $\underline{{\rm{Perf}}}AStack^{i}_{S}$, ${\rm{\underline{QPerf}}}AStack^{i}_{S}$, ${\rm{\underline{SPerf}}}AStack^{i}_{S}$, and ${\rm{\underline{StPerf}}}AStack^{i}_{S}$ of perfect algebraic stacks are all stable under $2$-fibre products. In particular, the $2$-categories ${\rm{Perf}}DM_{S}$, ${\rm{\underline{Perf}}}DM_{S}$, ${\rm{\underline{QPerf}}}DM^{i}_{S}$, ${\rm{\underline{SPerf}}}DM^{i}_{S}$, and ${\rm{\underline{StPerf}}}DM^{i}_{S}$ of perfect Deligne-Mumford stacks are also stable under $2$-fibre products.
\end{proposition}
More importantly, we formalize notion of characteristic of an algebraic stack. Given an algebraic stack $\calx$ in characteristic $p$, there is a canonical map $\calx\rightarrow\calx$, which is called the \textit{algebraic Frobenius morphism} of $\calx$. Although our definition of perfect algebraic stacks is completely separated from Frobenius morphisms, it recovers the following desirable statement.
\begin{theorem}
Let $\calx$ be an algebraic stack in characteristic $p$ over $S$ with algebraic Frobenius $\Psi_{\calx}:\calx\rightarrow\calx$. Then $\calx$ is perfect if and only if $\Psi_{\calx}$ is an equivalence.
\end{theorem}
Every algebraic stack $\calx$ in characteristic $p$ over $S$ has a perfection $\calx^{pf}$. The algebraic Frobenius morphism provides us with an explicit description of the perfection of $\calx$.
\begin{proposition}
Let $\calx$ be an algebraic stack in characteristic $p$ over $S$ with perfection $\calx^{pf}$. Then there is an equivalence
$$
\calx^{pf}\cong\lim_{\substack{\longleftarrow \\ \Psi_{\calx}}}\calx.
$$
\end{proposition}

Let $AStack^{p}$ be the 2-category of algebraic stacks $\calx$ over $S$ with ${\rm{char}}(\calx)=p$. The perfection of algebraic stacks naturally gives rise to a 2-functor
$$
{\rm{\underline{Perf}}}_{S}:AStack^{p}\longrightarrow\textrm{Perf}AStack_{S}.
$$
Such a 2-functor is called the \textit{perfection $2$-functor}. Note that the perfection $2$-functor induces an ordinary functor ${\rm{\underline{Perf}}}_{S}$ called the \textit{perfection functor}. They satisfy the following desirable properties.
\begin{proposition}
Here is a list of properties of ${\rm{\underline{Perf}}}_{S}$:
\begin{enumerate}
  \item
The perfection functor ${\rm{\underline{Perf}}}_{S}$ is full.
  \item
The perfection functor ${\rm{\underline{Perf}}}_{S}$ has a right adjoint.
\item
The perfection functor ${\rm{\underline{Perf}}}_{S}$ is left exact, and thus commutes with fibre products.
  \item
The perfection $2$-functor ${\rm{\underline{Perf}}}_{S}$ commutes with $2$-fibre products.
\end{enumerate}
\end{proposition}

The perfection $2$-functor enables us to pass the properties between the usual world and the perfect world.
\begin{theorem}\label{TT1}
Let $f:\calx\rightarrow\caly$ be a $1$-morphism of algebraic stacks in characteristic $p$ over $S$. Let $f^{\natural}:\calx^{pf}\rightarrow\caly^{pf}$ be the image of $f$ under the perfection $2$-functor $\underline{\rm{Perf}}_{S}$. Then
\begin{enumerate}
  \item
$f$ is representable by algebraic spaces, then $f^{\natural}$ is perfect;
  \item
$f$ has property $\cal{P}$, then $f^{\natural}$ has property $\cal{P}$;
\item
$f$ has property $\cal{P}'$ if and only if $f^{\natural}$ has property $\cal{P}'$;
  \item
$\calx$ has property $\cal{P}$, then $\calx^{pf}$ has property $\cal{P}$.
\end{enumerate}
\end{theorem}
\begin{remark}
Here $\cal{P},\cal{P}'$ are properties of algebraic spaces or morphisms of algebraic spaces satisfying some extra conditions.
\end{remark}

By means of Theorem \ref{TT1}, we show that the algebraic Frobenius shares the same properties as the absolute Frobenius of schemes.
\begin{proposition}
Let $\calx$ be an algebraic stack in characteristic $p$ over $S$ with algebraic Frobenius $\Psi_{\calx}:\calx\rightarrow\calx$. Then the following statements are satisfied:
\begin{enumerate}
  \item
$\Psi_{\calx}$ is representable by algebraic spaces.
  \item
$\Psi_{\calx}$ is integral, and is a universal homeomorphism.
  \item
$\Psi_{\calx}$ is surjective.
  \item
If $\calx$ is perfect, then $\Psi_{\calx}$ is perfect.
\end{enumerate}
\end{proposition}

Finally, we compare our perfect algebraic stacks with Zhu's perfect algebraic stacks. Let $k$ be a perfect field of characteristic $p$ and let $\mathcal{Z}AS_{k}^{pf}$ be the $2$-category of Zhu's perfect algebraic stacks over $k$. We show that there is a string of inclusions
\begin{align}
\mathcal{Z}AS_{k}^{pf}\subset{\rm{Perf}}AStack_{k}\subset\underline{\textrm{Perf}}AStack^{2}_{k}\subset\underline{\mathcal{Q}\textrm{Perf}}AStack^{2}_{k}\subset\underline{\mathcal{S}\textrm{Perf}}AStack^{2}_{k}\subset\underline{\mathcal{ST}\textrm{Perf}}AStack^{2}_{k}.
\end{align}
In other words, our perfect algebraic stacks generalize Zhu's perfect algebraic stacks.

\subsection{Outline}
In \S\ref{B1}, we develop a theory of perfect categories fibred in groupoids and perfect $1$-morphisms. This material will be applied to the following sections. Next, in \S\ref{B2}, we begin by introducing the definitions of all kinds of perfect algebraic stacks. We show that all these perfect algebraic stacks are stable under fibre products.

In \S\ref{B3}, we first formalize the characteristic of an algebraic stack. Then we study the properties of the canonical morphism of an algebraic stack. Most importantly, we deduce the algebraic Frobenius of an algebraic stack and prove the main Theorem \ref{T3}.

In \S\ref{B4}, we construct the perfection of an algebraic stack in characteristic $p$. Then by means of algebraic Frobenius morphisms, we show that every perfection can be described as an inverse limit. Then there is a natural notion of perfection 2-functor. We show that the perfection 2-functor enjoys several desirable properties.

Finally, in \S\ref{B5}, we compare our theory of perfect algebraic stacks with Zhu's one. At the end, we provide an equivalence between our perfect algebraic stacks and Zhu's perfect algebraic stacks.

\subsection{Conventions}
Throughout this paper, $p,q,\ell$ will always be prime numbers. The set of natural numbers will be $\N=\{0,1,2,...\}$. We will make use of the conventions in \cite[Tag04XA]{Stack Project} as follows:
\begin{enumerate}
  \item
Without explicitly mentioned, all schemes will be contained in $Sch_{fppf}$.
  \item
$S$ will always be a fixed base scheme contained in the big fppf site $Sch_{fppf}$.
  \item
Sometimes, we will not distinguish between a scheme $U$ (resp. an algebraic space $X$) and the algebraic stack $\SchU\rightarrow\SchS$ (resp. $\cals_{X}\rightarrow\SchS$).
\item
Sometimes, we will abbreviate $1$-morphism to morphism, namely we may say $f:\calx\rightarrow\caly$ is a morphism of algebraic stacks to indicate $f$ is a $1$-morphism of algebraic stacks over base scheme $S$.
\end{enumerate}

We will make use of the definition of algebraic stacks in \cite[Tag026O]{Stack Project}. By an \textit{algebraic stack $X$ over} $S$, we mean a stack in groupoids $X$ over $\SchS$ whose diagonal is representable by algebraic spaces, and it admits a surjective smooth map $U\rightarrow X$ from a scheme $U$.

In \cite{Liang}, one has the notions of perfect, quasi-perfect, semiperfect, and strongly perfect algebraic spaces. However, for simplicity, we will sometimes use ``perfect algebraic spaces'' to mean these four types of spaces. Let $X,Y$ be two categories. Without explicitly mentioned, the notation $X\cong Y$ will mean that $X$ is equivalent to $Y$.

\section{Some preliminary of categories fibred in groupoids}\label{B1}
In this section, we will develop some necessary materials concerning categories fibred in groupoids and perfect schemes (algebraic spaces). All these material will be useful in the following sections. We first focus on categories fibred in groupoids that are representable by a scheme or an algebraic space.

Recall that a category fibred in groupoids $\calx$ over $\SchS$ is said to be \textit{representable} (resp. \textit{representable by an algebraic space}) if there is an equivalence $\calx\cong\SchU$ (resp. $\calx\cong\cals_{F}$) of categories over $\SchS$ for some scheme $U\in\ObSchS$ (resp. for some algebraic space $F$ over $S$). We specialize these definitions and make the following definitions.
\begin{definition}
Let $\calx$ be a category fibred in groupoids over $\SchS$.
\begin{enumerate}
\item
We say that $\calx$ is weakly perfect if there exists a perfect scheme $U\in\ObSchS$ such that there is an equivalence $\calx\cong\SchU$ of categories over $\SchS$. In other words, $\calx$ is weakly perfect if it is representable by a perfect scheme.
\item
We say that $\calx$ is perfect (resp. quasi-perfect, resp. semiperfect, resp. strongly perfect) if there exists a perfect (resp. quasi-perfect, resp. semiperfect, resp. strongly perfect) algebraic space $F$ over $S$ such that there is an equivalence $\calx\cong\cal{S}_{F}$ of categories over $\SchS$. In other words, $\calx$ is perfect (resp. quasi-perfect, resp. semiperfect, resp. strongly perfect) if it is representable by a perfect (resp. quasi-perfect, resp. semiperfect, resp. strongly perfect) algebraic space.
\item
We say that $\calx$ is representably perfect if there exists a representable perfect algebraic space $F$ over $S$ such that there is an equivalence $\calx\cong\cal{S}_{F}$ of categories over $\SchS$. In other words, $\calx$ is representably perfect if it is representable by a representable perfect algebraic space.
\end{enumerate}
\end{definition}
\begin{remark}
Note that we only define representably perfect categories fibred in groupoids, since the other corresponding notions are trivial due to \cite[Lemma 3.3]{Liang}.
\end{remark}

It is obvious that every weakly perfect category fibred in groupoids is perfect. Meanwhile, every strongly perfect category fibred in groupoids is both quasi-perfect and semiperfect. And every quasi-perfect category fibred in groupoids is semiperfect.

Here is the lemma that characterizes weakly perfect, perfect, quasi-perfect, and semiperfect categories fibred in groupoids.
\begin{lemma}\label{L3}
Let $\calx$ be a category fibred in groupoids over $\SchS$. Then
\begin{enumerate}
\item
$\calx$ is weakly perfect if and only if $\calx$ is fibred in setoids, and the presheaf $U\mapsto\Ob(\calx_{U})/\cong$ is a perfect algebraic space.
\item
$\calx$ is perfect if and only if $\calx$ is fibred in setoids, and the presheaf $U\mapsto\Ob(\calx_{U})/\cong$ is a perfect algebraic space.
\item
$\calx$ is quasi-perfect if and only if $\calx$ is fibred in setoids, and the presheaf $U\mapsto\Ob(\calx_{U})/\cong$ is a quasi-perfect algebraic space.
\item
$\calx$ is semiperfect if and only if $\calx$ is fibred in setoids, and the presheaf $U\mapsto\Ob(\calx_{U})/\cong$ is a semiperfect algebraic space.
\item
$\calx$ is strongly perfect if and only if $\calx$ is fibred in setoids, and the presheaf $U\mapsto\Ob(\calx_{U})/\cong$ is a strongly perfect algebraic space.
\end{enumerate}
\end{lemma}
\begin{proof}
(1): Assume that $\calx$ is weakly perfect. It follows from \cite[Tag0045]{Stack Project} that the presheaf $U\mapsto\Ob(\calx_{U})/\cong$ is represented by a perfect scheme. Hence, by \cite[Lemma 3.3]{Liang}, the presheaf is a perfect algebraic space. Conversely, it follows from \cite[Tag0045]{Stack Project} again that there is an equivalence $\calx\rightarrow\SchU$ for some perfect scheme $U\in\ObSchS$ such that $\calx$ is weakly perfect.

The rest are similar to (1) as they can also be deduced from \cite[Tag0045]{Stack Project} and \cite[Lemma 3.3]{Liang}.
\end{proof}

Next, we focus on $1$-morphisms between categories fibred in groupoids that are representable or representable by algebraic spaces. Recall that a $1$-morphism $f:\calx\rightarrow\caly$ of categories fibred in groupoids over $\SchS$ is \textit{representable} (resp. \textit{representable by algebraic spaces}) if for any $U\in\ObSchS$ and any $\SchU\rightarrow\caly$, the category fibred in groupoids $\SchU\times_{\caly}\calx$ over $\SchU$ is a scheme (resp. an algebraic space) over $U$.

We then specialize these notions to the following cases.
\begin{definition}\label{A2}
Let $f:\calx\rightarrow\caly$ be a $1$-morphism of categories fibred in groupoids over $\SchS$. Then
\begin{enumerate}
\item
$f$ is said to be {\rm{$0$-weakly perfect}} if there exists $U\in\ObSchS$ such that for any $y\in{\rm{Ob}}(\caly_{U})$, the category fibred in groupoids $\SchU\times_{y,\caly}\calx$ is representable by a perfect scheme over $U$.
\item
$f$ is said to be {\rm{$1$-weakly perfect}} if there exists perfect scheme $U\in\ObSchS$ such that for any $y\in{\rm{Ob}}(\caly_{U})$, the category fibred in groupoids $\SchU\times_{y,\caly}\calx$ is representable by a perfect scheme over $U$.
\item
$f$ is said to be {\rm{$2$-weakly perfect}} if for every perfect scheme $U\in\ObSchS$ and any $y\in{\rm{Ob}}(\caly_{U})$, the category fibred in groupoids $\SchU\times_{y,\caly}\calx$ is representable by a perfect scheme over $U$.
\item
$f$ is said to be {\rm{$3$-weakly perfect}} if for every $U\in\ObSchS$ and any $y\in{\rm{Ob}}(\caly_{U})$, the category fibred in groupoids $\SchU\times_{y,\caly}\calx$ is representable by a perfect scheme over $U$. In other words, $f$ is {\rm{$3$-weakly perfect}} if it is representable by perfect schemes.
\item
$f$ is said to be $0$-perfect (resp.$0$-quasiperfect, resp. $0$-semiperfect, resp. $0$-strongly perfect) if there exists $U\in\ObSchS$ such that for any $y\in{\rm{Ob}}(\caly_{U})$, the category fibred in groupoids $\SchU\times_{y,\caly}\calx$ is representable by a perfect (resp. quasi-perfect, resp. semiperfect, resp. strongly perfect) algebraic space over $U$.
\item
$f$ is said to be $1$-perfect (resp. $1$-quasiperfect, resp. $1$-semiperfect, resp. $1$-strongly perfect) if there exists perfect scheme $U\in\ObSchS$ such that for any $y\in{\rm{Ob}}(\caly_{U})$, the category fibred in groupoids $\SchU\times_{y,\caly}\calx$ is representable by a perfect (resp. quasi-perfect, resp. semiperfect, resp. strongly perfect) algebraic space over $U$.
\item
$f$ is said to be $2$-perfect (resp. $2$-quasiperfect, resp. $2$-semiperfect, resp. $2$-strongly perfect) if for every perfect scheme $U\in\ObSchS$ and any $y\in{\rm{Ob}}(\caly_{U})$, the category fibred in groupoids $\SchU\times_{y,\caly}\calx$ is representable by a perfect (resp. quasi-perfect, resp. semiperfect, resp. strongly perfect) algebraic space over $U$.
\item
$f$ is said to be $3$-perfect (resp. $3$-quasiperfect, resp. $3$-semiperfect, resp. $3$-strongly perfect) if for every $U\in\ObSchS$ and any $y\in{\rm{Ob}}(\caly_{U})$, the category fibred in groupoids $\SchU\times_{y,\caly}\calx$ is representable by a perfect (resp. quasi-perfect, resp. semiperfect, resp. strongly perfect) algebraic space over $U$. In other words, $f$ is $3$-perfect (resp. $3$-quasiperfect, resp. $3$-semiperfect, resp. $3$-strongly perfect) if it is representable by perfect (resp. quasi-perfect, resp. semiperfect, resp. strongly perfect) algebraic spaces.
\end{enumerate}
\end{definition}
\begin{remark}
We will simply call \textit{weakly perfect} (resp. \textit{perfect}, resp. \textit{quasiperfect}, resp. \textit{semiperfect}, resp. \textit{strongly perfect}) $1$-morphism for any $0$-weakly perfect (resp. $0$-perfect, resp. $0$-quasiperfect, resp. $0$-semiperfect, resp. $0$-strongly perfect) $1$-morphism, when there is no confusion.
\end{remark}

Clearly, every $1$-perfect $1$-morphism is $0$-perfect. Every $2$-perfect $1$-morphism is $1$-perfect. And every $3$-perfect $1$-morphism is $2$-perfect. Let $0\leq i\leq3$ be an integer. These hold similarly for $i$-weakly perfect, $i$-quasiperfect, $i$-semiperfect, and $i$-strongly perfect $1$-morphisms.

We specialize the definitions of $i$-perfect $1$-morphisms to the following case of representable functors.
\begin{definition}
Let $f:\calx\rightarrow\caly$ be a $1$-morphism of categories fibred in groupoids over $\SchS$ that is representable by algebraic spaces. Let $U\in\ObSchS$ and let $y\in{\rm{Ob}}(\caly_{U})$ such that the category fibred in groupoids $\SchU\times_{y,\caly}\calx$ is representable by an algebraic space $F$ over $U$.
\begin{enumerate}
\item
Such an $F$ is called an associated algebraic space of $f$.
\item
If $f$ is $i$-perfect such that the category fibred in groupoids $\SchU\times_{y,\caly}\calx$ is representable by a perfect algebraic space $F$ over $U$. Then such an $F$ is called an associated $i$-perfect algebraic space of $f$. And $f$ is said to be representably $i$-perfect if all associated $i$-perfect algebraic spaces of $f$ are representable.
\end{enumerate}
\end{definition}
Note that every $i$-weakly perfect $1$-morphism is representably $i$-perfect. The following lemma characterizes some 1-morphisms defined above. And it is useful in the next section.
\begin{lemma}\label{L1}
Let $f:\calx\rightarrow\caly$ be a $1$-morphism of categories fibred in groupoids over $\SchS$. Suppose that $\caly$ is representable by an algebraic space.
\begin{enumerate}
\item
If $\calx$ is perfect, then the $1$-morphism $f$ is $2$-perfect. In particular, if $\calx$ is weakly perfect, then the $1$-morphism $f$ is $2$-weakly perfect.
\item
If $\calx$ is quasi-perfect, then the $1$-morphism $f$ is $2$-quasiperfect.
\item
If $\calx$ is semiperfect, then the $1$-morphism $f$ is $2$-semiperfect.
\item
If $\calx$ is strongly perfect, then the $1$-morphism $f$ is $2$-strongly perfect.
\item
If $\calx$ and $\caly$ are representably perfect, then the $1$-morphism $f$ is representably $2$-perfect.
\end{enumerate}
\end{lemma}
\begin{proof}
(1): Choose $\calx\cong\cal{S}_{F},\caly\cong\cal{S}_{G}$ where $F$ is a perfect algebraic space and $G$ is an algebraic space. Assume that $U\in\ObSchS$ is a perfect scheme and that $y\in{\rm{Ob}}(\caly_{U})$. Then $\SchU\times_{y,\caly}\calx\cong\cals_{h_{U}}\times_{\cals_{G}}\cals_{F}=\cals_{h_{U}\times_{G}F}$. By \cite[Proposition 3.7]{Liang}, $h_{U}\times_{G}F$ is a perfect algebraic space. Thus, $\SchU\times_{y,\caly}\calx$ is representable by a perfect algebraic space. And hence, the 1-morphism $f$ is $2$-perfect.

(2): Let $F$ be a quasi-perfect algebraic space and $G$ be an algebraic space such that $\calx\cong\cal{S}_{F},\caly\cong\cal{S}_{G}$. Assume that $U\in\ObSchS$ is a perfect scheme and that $y\in{\rm{Ob}}(\caly_{U})$. Then $\SchU\times_{y,\caly}\calx\cong\cals_{h_{U}}\times_{\cals_{G}}\cals_{F}=\cals_{h_{U}\times_{G}F}$. By \cite[Proposition 3.9]{Liang}, $h_{U}\times_{G}F$ is a quasi-perfect algebraic space. Thus, $\SchU\times_{y,\caly}\calx$ is representable by a quasi-perfect algebraic space. And hence, the 1-morphism $f$ is $2$-quasiperfect.

The rest are all similar to the proof of (1) and (2).
\end{proof}

In the following, we have a series of propositions which characterize categories fibred in groupoids whose diagonal morphisms are $i$-perfect (resp. $i$-quasiperfect, resp. $i$-semiperfect, resp. $i$-strongly perfect).
\begin{proposition}\label{P1}
Let $\calx$ be a category fibred in groupoids over $\SchS$. The following are equivalent:
\begin{itemize}
\item[(1)]
The diagonal $\Delta:\calx\rightarrow\calx\times\calx$ is $3$-perfect (resp. $3$-quasiperfect, resp. $3$-semiperfect, resp. $3$-strongly perfect).
\item[(2)]
For every $U\in\ObSchS$, and any $x,y\in{\rm{Ob}}(\calx_{U})$, the presheaf $\textit{Isom}(x,y)$ is a perfect (resp. quasi-perfect, resp. semiperfect, resp. strongly perfect) algebraic space.
\item[(3)]
For every $U\in\ObSchS$, and any $x\in{\rm{Ob}}(\calx_{U})$, the associated $1$-morphism $x:\SchU\rightarrow\calx$ is $3$-perfect (resp. $3$-quasiperfect, resp. $3$-semiperfect, resp. $3$-strongly perfect).
\end{itemize}
\end{proposition}
\begin{proof}
We just prove the first case as the rest are similar.

$(1)\Leftrightarrow(2):$ Assume that the diagonal $\calx\rightarrow\calx\times\calx$ is $3$-perfect. Let $\calx\times\calx=\caly$. Let $U\in\ObSchS$ and $y\in{\rm{Ob}}(\caly_{U})$. The category $\SchU\times_{y,\caly}\calx$ is representable by a perfect algebraic space and hence is perfect. Let $x',y'\in{\rm{Ob}}(\calx_{U})$. By \cite[Tag04SI]{Stack Project} the presheaf $\textit{Isom}(x',y')$ is coincident with the presheaf given by $U\mapsto\Ob(\calx_{U})/\cong$. Now, Lemma \ref{L3} shows that $\textit{Isom}(x',y')$ is a perfect algebraic space.

Conversely, let $x',y'\in{\rm{Ob}}(\calx_{U})$. Note that the category $\SchU\times_{y,\caly}\calx$ is fibred in groupoids over $\SchU$. Then $\SchU\times_{y,\caly}\calx$ is fibred in setoids over $\SchU$ and the presheaf $\textit{Isom}(x',y')$ is coincident with the presheaf given by $U\mapsto\Ob(\calx_{U})/\cong$ due to \cite[Tag04SI]{Stack Project}. Now, by Lemma \ref{L3}, $\SchU\times_{y,\caly}\calx$ is perfect.

$(1)\Leftrightarrow(3):$ Assume (1). For any $V\in\ObSchS$ and $y\in\Ob(\calx_{V})$, we see that
$$
\SchU\times_{x,\calx,y}\SchV\cong((Sch/U\times_{S}V)_{fppf})\times_{(x,y),\calx\times\calx,\Delta}\calx\cong\cals_{F},
$$
where $F$ is a perfect algebraic space by assumption. Hence, the associated $1$-morphism $x:\SchU\rightarrow\calx$ is $3$-perfect.

Conversely, assume (3). For any pair of objects $x,x'\in\Ob(\calx_{U})$, we have
$$
\calx\times_{\Delta,\calx\times\calx,(x,x')}\SchU\cong(\SchU\times_{x,\calx,x'}\SchU)\times_{(Sch/U\times_{S}U)_{fppf},\Delta}\SchU.
$$
By assumption, the associated $1$-morphism $x:\SchU\rightarrow\calx$ is $3$-perfect. Thus, the right hand side above is represented by a perfect algebraic space $F$, which implies that $\calx\times_{\Delta,\calx\times\calx,(x,x')}\SchU\cong\cals_{F}$ such that the diagonal $\Delta:\calx\rightarrow\calx\times\calx$ is $3$-perfect.
\end{proof}

\begin{proposition}
Let $\calx$ be a category fibred in groupoids over $\SchS$. The following are equivalent:
\begin{enumerate}
\item
The diagonal $\calx\rightarrow\calx\times\calx$ is $2$-perfect (resp. $2$-quasiperfect, resp. $2$-semiperfect, resp. $2$-strongly perfect).
\item
For every perfect scheme $U\in\ObSchS$, and any $x,y\in{\rm{Ob}}(\calx_{U})$, the presheaf $\textit{Isom}(x,y)$ is a perfect (resp. quasi-perfect, resp. semiperfect, resp. strongly perfect) algebraic space.
\item
For every perfect scheme $U\in\ObSchS$, and any $x\in{\rm{Ob}}(\calx_{U})$, the associated $1$-morphism $x:\SchU\rightarrow\calx$ is $2$-perfect (resp. $2$-quasiperfect, resp. $2$-semiperfect, resp. $2$-strongly perfect).
\end{enumerate}
\end{proposition}
\begin{proof}
The proof is quite similar to that of Proposition \ref{P1}.
\end{proof}

\begin{proposition}
Let $\calx$ be a category fibred in groupoids over $\SchS$. The following are equivalent:
\begin{itemize}
\item[(1)]
The diagonal $\calx\rightarrow\calx\times\calx$ is $1$-perfect (resp. $1$-quasiperfect, resp. $1$-semiperfect, resp. $1$-strongly perfect).
\item[(2)]
There is perfect scheme $U\in\ObSchS$ such that for any $x,y\in{\rm{Ob}}(\calx_{U})$, the presheaf $\textit{Isom}(x,y)$ is a perfect (resp. quasi-perfect, resp. semiperfect, resp. strongly perfect) algebraic space.
\item[(3)]
There is perfect scheme $U\in\ObSchS$ such that for any $x\in{\rm{Ob}}(\calx_{U})$, the associated $1$-morphism $x:\SchU\rightarrow\calx$ is $1$-perfect (resp. $1$-quasiperfect, resp. $1$-semiperfect, resp. $1$-strongly perfect).
\end{itemize}
\end{proposition}
\begin{proof}
The proof is quite similar to that of Proposition \ref{P1}.
\end{proof}

\begin{proposition}\label{P3}
Let $\calx$ be a category fibred in groupoids over $\SchS$. The following are equivalent:
\begin{itemize}
\item[(1)]
The diagonal $\calx\rightarrow\calx\times\calx$ is perfect (resp. quasiperfect, resp. semiperfect, resp. strongly perfect).
\item[(2)]
There is $U\in\ObSchS$ such that for any $x,y\in{\rm{Ob}}(\calx_{U})$, the presheaf $\textit{Isom}(x,y)$ is a perfect (resp. quasiperfect, resp. semiperfect, resp. strongly perfect) algebraic space.
\item[(3)]
There is $U\in\ObSchS$ such that for any $x\in{\rm{Ob}}(\calx_{U})$, the associated $1$-morphism $x:\SchU\rightarrow\calx$ is perfect (resp. quasiperfect, resp. semiperfect, resp. strongly perfect).
\end{itemize}
\end{proposition}
\begin{proof}
The proof is quite similar to that of Proposition \ref{P1}.
\end{proof}

The following lemma shows that one can pass between $i$-perfect (resp. $i$-weakly perfect, resp. $i$-quasiperfect, resp. $i$-semiperfect, resp. $i$-strongly perfect) $1$-morphisms in any $2$-commutative diagram as follows.
\begin{lemma}\label{L5}
Consider the $2$-commutative diagram
$$
\xymatrix{
  \calx' \ar[d]_{f'} \ar[r]^{} & \calx \ar[d]^{f} \\
  \caly' \ar[r]^{} & \caly   }
$$
of $1$-morphisms over $\SchS$, where the horizontal arrows are equivalences. Then
\begin{enumerate}
\item
$f$ is $i$-weakly perfect if and only if $f'$ is $i$-weakly perfect.
\item
$f$ is $i$-perfect if and only if $f'$ is $i$-perfect. In particular, $f$ is representably $i$-perfect if and only if $f'$ is representably $i$-perfect.
\item
$f$ is $i$-quasiperfect if and only if $f'$ is $i$-quasiperfect.
\item
$f$ is $i$-semiperfect if and only if $f'$ is $i$-semiperfect.
\item
$f$ is $i$-strongly perfect if and only if $f'$ is $i$-strongly perfect.
\end{enumerate}
\end{lemma}
\begin{proof}
We just prove (2) since the other are similar. Let $U\in\ObSchS$. Then we have the following equivalences
$$
\SchU\times_{\caly}\calx\cong\SchU\times_{\caly'}\calx'\cong\cals_{F},
$$
where $F$ is an associated $i$-perfect algebraic space over $U$. This proves (2).
\end{proof}

Next, one can show that $i$-perfect (resp. $i$-weakly perfect, resp. $i$-quasiperfect, resp. $i$-semiperfect, resp. $i$-strongly perfect) $1$-morphisms are stable under compositions.
\begin{proposition}\label{P2}
Let $f:\calx\rightarrow\caly$ and $g:\caly\rightarrow\calz$ be two $1$-morphisms of categories fibred in groupoids over $\SchS$.
\begin{enumerate}
\item
If $f$ and $g$ are $i$-weakly perfect, then $g\circ f$ is $i$-weakly perfect.
\item
If $f$ and $g$ are $i$-perfect, then $g\circ f$ is $i$-perfect. In particular, if $f,g$ are representably $i$-perfect, then $g\circ f$ is representably $i$-perfect.
\item
If $f$ and $g$ are $i$-quasiperfect, then $g\circ f$ is $i$-quasiperfect.
\item
If $f$ and $g$ are $i$-semiperfect, then $g\circ f$ is $i$-semiperfect.
\item
If $f$ and $g$ are $i$-strongly perfect, then $g\circ f$ is $i$-strongly perfect.
\end{enumerate}
\end{proposition}
\begin{proof}
(2): Without loss of generality, we take $i=0$. Let $U,U'\in\ObSchS$. Since $f,g$ are perfect, there is some perfect algebraic spaces $F,F'$ such that $\SchU\times_{\caly}\calx\cong\cals_{F}$ and $\SchUU\times_{\calz}\caly\cong\cals_{F'}$. Now, $(\SchU\times_{\caly}\calx)\times_{\SchUU}(\SchUU\times_{\calz}\caly)\cong\cals_{F}\times_{\cals_{h_{U'}}}\cals_{F'}=\cals_{F\times_{h_{U'}}F'}$. But $(\SchU\times_{\caly}\calx)\times_{\SchUU}(\SchUU\times_{\calz}\caly)\cong\SchU\times_{\caly}\calx\times_{\calz}\caly\cong\SchU\times_{\calz}\calx$. Hence, $\SchU\times_{\calz}\calx\cong\cals_{F\times_{h_{U'}}F'}$. By \cite[Proposition 3.7]{Liang} the algebraic space $F\times_{h_{U'}}F'$ is perfect. Thus, $g\circ f$ is perfect. The second statement is clear.

(3): Without loss of generality, we take $i=0$. Let $U,U'\in\ObSchS$ such that there exist quasi-perfect algebraic spaces $F,F'$ such that $\SchU\times_{\caly}\calx\cong\cals_{F}$ and $\SchUU\times_{\calz}\caly\cong\cals_{F'}$. Then as (1) above we have $\SchU\times_{\calz}\calx\cong\cals_{F\times_{h_{U'}}F'}$. Now, \cite[Proposition 3.9]{Liang} shows that the algebraic space $F\times_{h_{U'}}F'$ is quasi-perfect. Thus, $g\circ f$ is quasi-perfect. The second statement is clear.

(1) is a special case of (2). (4), (5) follow from (3) by replacing ``quasi-perfect'' by ``semiperfect/strongly perfect''.
\end{proof}

Moreover, $i$-perfect (resp. $i$-weakly perfect, resp. $i$-quasiperfect, resp. $i$-semiperfect, resp. $i$-strongly perfect) $1$-morphisms are stable under arbitrary base change.
\begin{proposition}
Let $\calx,\caly,\calz$ be categories fibred in groupoids over $\SchS$. Consider the $2$-fibre product diagram
$$
\xymatrix{
  \calx\times_{\caly}\calz \ar[d]_{a'} \ar[rr]^{b'} &   & \calx \ar[d]^{a} \\
  \calz \ar[rr]^{b} &  & \caly   }
$$
\begin{enumerate}
\item
If $a$ is $i$-weakly perfect, then $a'$ is $i$-weakly perfect.
\item
If $a$ is $i$-perfect, then $a'$ is $i$-perfect. In particular, if $a$ is representably $i$-perfect, then $a'$ is representably $i$-perfect.
\item
If $a$ is $i$-quasiperfect, then $a'$ is $i$-quasiperfect.
\item
If $a$ is $i$-semiperfect, then $a'$ is $i$-semiperfect.
\item
If $a$ is $i$-strongly perfect, then $a'$ is $i$-strongly perfect.
\end{enumerate}
\end{proposition}
\begin{proof}
We will merely prove (2) as the rest are similar. Without loss of generality, take $i=0$. Let $U\in\ObSchS$. Since $a$ is perfect, there is some perfect algebraic space $F$ over $U$ such that $\SchU\times_{\caly}\calx\cong\cals_{F}$. Now,
$$
\SchU\times_{\calz}(\calx\times_{\caly}\calz)\cong\SchU\times_{\caly}\calx\cong\cals_{F}.
$$
This shows that $a'$ is perfect. If $F$ is representable, then $a'$ is representably perfect.
\end{proof}

Proposition \ref{P2} implies that we can form certain kinds of 2-categories. We first check the following lemma.
\begin{lemma}\label{L9}
Let $f:\calx\rightarrow\caly$ be a $1$-morphism of categories fibred in groupoids over $\SchS$ which is an equivalence. Then $f$ is $2$-weakly perfect. In particular, if $1_{\calx}:\calx\rightarrow\calx$ is the identity $1$-morphism, then $1_{\calx}$ is $2$-weakly perfect.
\end{lemma}
\begin{proof}
Let $U\in\ObSchS$ be a perfect scheme. Then note that there is an equivalence $\SchU\times_{\caly}\calx\cong\SchU$.
\end{proof}

We can define a list of full sub 2-categories of the 2-category of categories fibred in groupoids over $\SchS$ as follows:
\begin{definition1}\label{D3}\
\begin{enumerate}
  \item
  The \textit{$2$-category of weakly perfect categories fibred in groupoids over} $\SchS$, which is denoted by ${\rm{WPerf}}Cat\mathscr{F}$. Its $1$-morphisms will be $1$-morphisms of weakly perfect categories fibred in groupoids over $\SchS$, which are automatically 2-weakly perfect due to Lemma \ref{L1}.
  \item
  The \textit{$2$-category of perfect categories fibred in groupoids over} $\SchS$, which is denoted by ${\rm{Perf}}Cat\mathscr{F}$. Its $1$-morphisms will be $1$-morphisms of perfect categories fibred in groupoids over $\SchS$, which are automatically 2-perfect due to Lemma \ref{L1}.
  \item
  The \textit{$2$-category of quasi-perfect categories fibred in groupoids over} $\SchS$, which is denoted by ${\rm{QPerf}}Cat\mathscr{F}$. Its $1$-morphisms will be $1$-morphisms of quasi-perfect categories fibred in groupoids over $\SchS$, which are automatically 2-quasiperfect due to Lemma \ref{L1}.
  \item
  The \textit{$2$-category of semiperfect categories fibred in groupoids over} $\SchS$, which is denoted by ${\rm{SPerf}}Cat\mathscr{F}$. Its $1$-morphisms will be $1$-morphisms of semiperfect categories fibred in groupoids over $\SchS$, which are automatically 2-semiperfect due to Lemma \ref{L1}.
  \item
  The \textit{$2$-category of strongly perfect categories fibred in groupoids over} $\SchS$, which is denoted by ${\rm{StPerf}}Cat\mathscr{F}$. Its $1$-morphisms will be $1$-morphisms of strongly perfect categories fibred in groupoids over $\SchS$, which are automatically 2-strongly perfect due to Lemma \ref{L1}.
\end{enumerate}
\end{definition1}

One readily obtains the following lemma which says that the 2-categories in Definition \ref{D3} are closed under equivalences.
\begin{lemma}\label{L4}
Let $\calx,\caly$ be categories fibred in groupoids over $\SchS$. Suppose that $\calx$ and $\caly$ are equivalent. Then
\begin{itemize}
\item[(1)]
$\calx$ is weakly perfect if and only if $\caly$ is weakly perfect.
\item[(2)]
$\calx$ is perfect (resp. representably perfect) if and only if $\caly$ is perfect (resp. representably perfect).
\item[(3)]
$\calx$ is quasi-perfect if and only if $\caly$ is quasi-perfect.
\item[(4)]
$\calx$ is semiperfect if and only if $\caly$ is semiperfect.
\item[(5)]
$\calx$ is strongly perfect if and only if $\caly$ is strongly perfect.
\end{itemize}
\end{lemma}
\begin{proof}
It is easy to check from definitions.
\end{proof}

Moreover, it is easy to see that 2-categories in Definition \ref{D3} are stable under fibre products.
\begin{proposition}\label{P4}
Let $f:\calx\rightarrow\calz$ and $g:\caly\rightarrow\calz$ be $1$-morphisms of categories fibred in groupoids over $\SchS$.
\begin{enumerate}
  \item
  If $\calx,\caly,\calz$ are weakly perfect, then the $2$-fibre product $\calx\times_{\calz}\caly$ is weakly perfect. It is also a $2$-fibre product in the $2$-category ${\rm{WPerf}}Cat\mathscr{F}$.
  \item
  If $\calx,\caly,\calz$ are perfect, then the $2$-fibre product $\calx\times_{\calz}\caly$ is perfect. It is also a $2$-fibre product in the $2$-category ${\rm{Perf}}Cat\mathscr{F}$.
  \item
  If $\calx,\caly,\calz$ are quasi-perfect, then the $2$-fibre product $\calx\times_{\calz}\caly$ is quasi-perfect. It is also a $2$-fibre product in the $2$-category ${\rm{QPerf}}Cat\mathscr{F}$.
  \item
  If $\calx,\caly,\calz$ are semiperfect, then the $2$-fibre product $\calx\times_{\calz}\caly$ is semiperfect. It is also a $2$-fibre product in the $2$-category ${\rm{SPerf}}Cat\mathscr{F}$.
  \item
  If $\calx,\caly,\calz$ are strongly perfect, then the $2$-fibre product $\calx\times_{\calz}\caly$ is strongly perfect. It is also a $2$-fibre product in the $2$-category ${\rm{StPerf}}Cat\mathscr{F}$.
\end{enumerate}
\end{proposition}
\begin{proof}
By \cite[Tag0041]{Stack Project}, the $2$-fibre product $\calx\times_{\calz}\caly$ is a category fibred in groupoids over $\SchS$. Thus, the statements follow directly from \cite[Proposition 3.7, Proposition 3.9]{Liang} and the fact that 2-categories in Definition \ref{D3} are full sub 2-category of the 2-category of categories fibred in groupoids over $\SchS$.
\end{proof}

\section{Perfect algebraic stacks}\label{B2}
In this section, we extend our results of perfect algebraic spaces to the setting of algebraic stacks. Here are the definitions of different kinds of perfect algebraic stacks. Let $0\leq i\leq3$ be an integer.

\begin{definition}\label{D1}
Let $\mathcal{X}$ be an algebraic stack over $S$.
\begin{enumerate}
\item
Then $\mathcal{X}$ is said to be perfect if there exist a perfect scheme $U\in{\rm{Ob}}((Sch/S)_{fppf})$ together with a surjective smooth $1$-morphism $(Sch/U)_{fppf}\rightarrow\mathcal{X}$.
\item
We say $\mathcal{X}$ is relatively $i$-perfect (resp. $i$-quasiperfect, resp. $i$-semiperfect, resp. $i$-strongly perfect) if the diagonal morphism $\Delta:\mathcal{X}\rightarrow\mathcal{X}\times\mathcal{X}$ is $i$-perfect (resp. $i$-quasiperfect, resp. $i$-semiperfect, resp. $i$-strongly perfect).
\item
We say that $\mathcal{X}$ is representably relatively $i$-perfect (resp. relatively $i$-weakly perfect) if the diagonal morphism $\Delta:\mathcal{X}\rightarrow\mathcal{X}\times\mathcal{X}$ is representably $i$-perfect (resp. $i$-weakly perfect).
\item
We say $\mathcal{X}$ is representably weakly perfect (resp. perfect, resp. quasi-perfect, resp. semiperfect, resp. strongly perfect) if it is weakly perfect (resp. perfect, resp. quasi-perfect, resp. semiperfect, resp. strongly perfect) as a category fibred in groupoids.
\end{enumerate}
\end{definition}

We specialize the above definitions to the case of Deligne-Mumford stacks.
\begin{definition}
Let $\mathcal{X}$ be a Deligne-Mumford stack over $S$ and let $\tau$ be any property of an algebraic stack over $S$ in Definition \ref{D1} (2)-(4). Then $\mathcal{X}$ is said to be perfect if there exist a perfect scheme $U\in{\rm{Ob}}((Sch/S)_{fppf})$ and a surjective \'{e}tale $1$-morphism $(Sch/U)_{fppf}\rightarrow\mathcal{X}$. And $\mathcal{X}$ is said to be $\tau$ if it is $\tau$ as an algebraic stack over $S$.
\end{definition}
\begin{remark}
For simplicity, we simply call \textit{DM stack} for a Deligne-Mumford stack.
\end{remark}

The following lemma shows that certain kinds of Deligne-Mumford stack defined above generalize the notion of perfect algebraic spaces.
\begin{lemma}\label{L2}
Let $F$ be an algebraic space over $S$.
\begin{itemize}
\item[(1)]
If $F$ is perfect, then the associated category fibred in groupoid $p:\cals_{F}\rightarrow\SchS$ is a perfect, relatively $2$-perfect, and representably perfect DM stack. Furthermore, if $F$ is representable, then $\cals_{F}$ is representably relatively $2$-perfect. And if $F$ is represented by a perfect scheme, then $\cals_{F}$ is representably weakly perfect and relatively $2$-weakly perfect.
\item[(2)]
If $F$ is quasi-perfect, then the associated category fibred in groupoid $p:\cals_{F}\rightarrow\SchS$ is a relative $2$-quasiperfect and representably quasi-perfect DM stack.
\item[(3)]
If $F$ is semiperfect, then the associated category fibred in groupoid $p:\cals_{F}\rightarrow\SchS$ is a relative $2$-semiperfect and representably semiperfect DM stack.
\item[(4)]
If $F$ is strongly perfect, then the associated category fibred in groupoid $p:\cals_{F}\rightarrow\SchS$ is a relative $2$-strongly perfect and representably strongly perfect DM stack.
\end{itemize}
\end{lemma}
\begin{proof}
(1): Let $U\in\ObSchS$ be a perfect scheme such that $h_{U}\rightarrow F$ is surjective \'{e}tale. It follows from \cite[Tag045A]{Stack Project} that the 1-morphism $\SchU\rightarrow\cals_{F}$ is surjective \'{e}tale. Hence, $\cals_{F}$ is perfect. Since $F$ is perfect, the diagonal morphism $\cals_{F}\rightarrow\cals_{F\times F}$ is $2$-perfect by Lemma \ref{L1}. Thus, the category fibred in groupoid $\cals_{F}$ is relatively $2$-perfect.

If $F$ is representable (resp. represented by a perfect scheme), then the diagonal morphism $\cals_{F}\rightarrow\cals_{F\times F}$ is representably $2$-perfect (resp. $2$-weakly perfect) by Lemma \ref{L1}. Thus, the category fibred in groupoid $\cals_{F}$ is representably relatively $2$-perfect (resp. relatively $2$-weakly perfect).

(2): Since $F$ is quasi-perfect, it follows from Lemma \ref{L1} that the diagonal morphism $\cals_{F}\rightarrow\cals_{F\times F}$ is $2$-quasiperfect. Thus, the category fibred in groupoid $\cals_{F}$ is relatively $2$-quasiperfect. It is obvious that $\cals_{F}$ is quasi-perfect.

The proof of (3) and (4) are similar to (1) and (2).
\end{proof}

\begin{remark}
Let ``perfect'' represents ``perfect (resp. quasiperfect, resp. semiperfect, resp. strongly perfect)'' for now. If $\calx$ is a relatively $3$-perfect algebraic stack over $S$, then by virtue of Lemma \ref{L2}, $\calx$ should be said to be \textit{pseudo-perfect}, which indicates that it does not generalize perfect algebraic spaces. Moreover, by abuse, when we speak of a relatively $0$-perfect algebraic stack $\calx$ over $S$, we will simply call it \textit{relatively perfect} algebraic stack over $S$.
\end{remark}

We can form a series of full sub $2$-categories of the 2-category $AStack_{S}$ (resp. $DM_{S}$) of algebraic stacks (resp. DM stacks) over $S$ as follows.
\begin{definition1}\label{D2}
We have the following full sub 2-categories of the 2-category $AStack_{S}$ (resp. $DM_{S}$).
\begin{enumerate}
\item
The \textit{$2$-category of perfect algebraic stacks} (resp. \textit{DM stacks}) \textit{over} $S$, which is denoted by $\textrm{Perf}AStack_{S}$ (resp. $\textrm{Perf}DM_{S}$).
\item
The \textit{$2$-category of relatively $i$-perfect algebraic stacks} (resp. \textit{DM stacks}) \textit{over} $S$, which is denoted by $\textrm{\underline{Perf}}AStack^{i}_{S}$ (resp. $\textrm{\underline{Perf}}DM^{i}_{S}$).
\item
The \textit{$2$-category of relatively $i$-quasiperfect algebraic stacks} (resp. \textit{DM stacks}) \textit{over} $S$, which is denoted by $\underline{\mathcal{Q}\textrm{Perf}}AStack^{i}_{S}$ (resp. $\underline{\mathcal{Q}\textrm{Perf}}DM^{i}_{S}$).
\item
The \textit{$2$-category of relatively $i$-semiperfect algebraic stacks} (resp. \textit{DM stacks}) \textit{over} $S$, which is denoted by $\underline{\mathcal{S}\textrm{Perf}}AStack^{i}_{S}$ (resp. $\underline{\mathcal{S}\textrm{Perf}}DM^{i}_{S}$).
\item
The \textit{$2$-category of relatively $i$-strongly perfect algebraic stacks} (resp. \textit{DM stacks}) \textit{over} $S$, which is denoted by $\underline{\mathcal{ST}\textrm{Perf}}AStack^{i}_{S}$ (resp. $\underline{\mathcal{ST}\textrm{Perf}}DM^{i}_{S}$).
\item
The \textit{$2$-category of representably weakly perfect DM stacks over} $S$, which is denoted by $\widehat{\textrm{WPerf}}DM_{S}$. Its $1$-morphisms will be $1$-morphisms of weakly perfect DM stacks over $\SchS$, which are automatically 2-weakly perfect due to Lemma \ref{L1}.
\item
The \textit{$2$-category of representably perfect DM stacks over} $S$, which is denoted by $\widehat{\textrm{Perf}}DM_{S}$. Its $1$-morphisms will be $1$-morphisms of perfect DM stacks over $\SchS$, which are automatically 2-perfect due to Lemma \ref{L1}.
\item
The \textit{$2$-category of representably quasi-perfect DM stacks over} $S$, which is denoted by $\widehat{\textrm{QPerf}}DM_{S}$. Its $1$-morphisms will be $1$-morphisms of quasi-perfect DM stacks over $\SchS$, which are automatically 2-quasiperfect due to Lemma \ref{L1}.
\item
The \textit{$2$-category of representably semiperfect DM stacks over} $S$, which is denoted by $\widehat{\textrm{SPerf}}DM_{S}$. Its $1$-morphisms will be $1$-morphisms of semiperfect DM stacks over $\SchS$, which are automatically 2-semiperfect due to Lemma \ref{L1}.
\item
The \textit{$2$-category of representably strongly perfect DM stacks over} $S$, which is denoted by $\widehat{\textrm{StPerf}}DM_{S}$. Its $1$-morphisms will be $1$-morphisms of strongly perfect DM stacks over $\SchS$, which are automatically 2-strongly perfect due to Lemma \ref{L1}.
\end{enumerate}
\end{definition1}

Then we have the following two commutative diagrams of inclusions of categories
\begin{center}
$$
\begin{tikzcd}[row sep=1.5em, column sep = 1.5em]
\underline{\textrm{Perf}}AStack^{0}_{S}  \arrow[r] & \underline{\mathcal{Q}\textrm{Perf}}AStack^{0}_{S}  \arrow[r] & \underline{\mathcal{S}\textrm{Perf}}AStack^{0}_{S} \arrow[r] & \underline{\mathcal{ST}\textrm{Perf}}AStack^{0}_{S} \\
\underline{\textrm{Perf}}AStack^{1}_{S} \arrow[u] \arrow[r] & \underline{\mathcal{Q}\textrm{Perf}}AStack^{1}_{S}\arrow[u]  \arrow[r] & \underline{\mathcal{S}\textrm{Perf}}AStack^{1}_{S} \arrow[u] \arrow[r] & \underline{\mathcal{ST}\textrm{Perf}}AStack^{1}_{S} \arrow[u] \\
\underline{\textrm{Perf}}AStack^{2}_{S} \arrow[u] \arrow[r] & \underline{\mathcal{Q}\textrm{Perf}}AStack^{2}_{S}\arrow[u]  \arrow[r] & \underline{\mathcal{S}\textrm{Perf}}AStack^{2}_{S} \arrow[u] \arrow[r] & \underline{\mathcal{ST}\textrm{Perf}}AStack^{2}_{S} \arrow[u] \\
\underline{\textrm{Perf}}AStack^{3}_{S} \arrow[u] \arrow[r] & \underline{\mathcal{Q}\textrm{Perf}}AStack^{3}_{S}\arrow[u]  \arrow[r] & \underline{\mathcal{S}\textrm{Perf}}AStack^{3}_{S} \arrow[u] \arrow[r] & \underline{\mathcal{ST}\textrm{Perf}}AStack^{3}_{S} \arrow[u]
\end{tikzcd}
$$
\captionof{figure}{The commutative diagram of inclusion functors}
\label{F1}
\end{center}
\begin{center}
$$
\begin{tikzcd}[row sep=1.5em, column sep = 1.5em]
\underline{\textrm{Perf}}DM^{0}_{S}  \arrow[r] & \underline{\mathcal{Q}\textrm{Perf}}DM^{0}_{S}  \arrow[r] & \underline{\mathcal{S}\textrm{Perf}}DM^{0}_{S} \arrow[r] & \underline{\mathcal{ST}\textrm{Perf}}DM^{0}_{S} \\
\underline{\textrm{Perf}}DM^{1}_{S} \arrow[u] \arrow[r] & \underline{\mathcal{Q}\textrm{Perf}}DM^{1}_{S}\arrow[u]  \arrow[r] & \underline{\mathcal{S}\textrm{Perf}}DM^{1}_{S} \arrow[u] \arrow[r] & \underline{\mathcal{ST}\textrm{Perf}}DM^{1}_{S} \arrow[u] \\
\underline{\textrm{Perf}}DM^{2}_{S} \arrow[u] \arrow[r] & \underline{\mathcal{Q}\textrm{Perf}}DM^{2}_{S}\arrow[u]  \arrow[r] & \underline{\mathcal{S}\textrm{Perf}}DM^{2}_{S} \arrow[u] \arrow[r] & \underline{\mathcal{ST}\textrm{Perf}}DM^{2}_{S} \arrow[u] \\
\underline{\textrm{Perf}}DM^{3}_{S} \arrow[u] \arrow[r] & \underline{\mathcal{Q}\textrm{Perf}}DM^{3}_{S}\arrow[u]  \arrow[r] & \underline{\mathcal{S}\textrm{Perf}}DM^{3}_{S} \arrow[u] \arrow[r] & \underline{\mathcal{ST}\textrm{Perf}}DM^{3}_{S} \arrow[u]
\end{tikzcd}
$$
\captionof{figure}{The commutative diagram of inclusion functors}
\label{F2}
\end{center}
together with a string of inclusion functors
\begin{align}
\widehat{\textrm{WPerf}}DM_{S}\subset\widehat{\textrm{Perf}}DM_{S}\subset\widehat{\textrm{QPerf}}DM_{S}\subset\widehat{\textrm{SPerf}}DM_{S}\subset\widehat{\textrm{StPerf}}DM_{S}.
\end{align}

The following lemma shows that all 2-categories in Definition \ref{D2} are closed under equivalences.
\begin{lemma}\label{L13}
Let $\calx,\caly$ be categories fibred in groupoids over $\SchS$. Suppose that $\calx$ and $\caly$ are equivalent. Then
\begin{enumerate}
\item
$\calx$ is a perfect algebraic stack (resp. DM stack) if and only if $\caly$ is a perfect algebraic stack (resp. DM stack).
\item
$\calx$ is a representably perfect DM stack if and only if $\caly$ is a representably perfect DM stack. In particular, $\calx$ is a representably weakly perfect DM stack if and only if $\caly$ is a representably weakly perfect DM stack.
\item
$\calx$ is a representably quasi-perfect DM stack if and only if $\caly$ is a representably quasi-perfect DM stack.
\item
$\calx$ is a representably semiperfect DM stack if and only if $\caly$ is a representably semiperfect DM stack.
\item
$\calx$ is a representably strongly perfect DM stack if and only if $\caly$ is a representably strongly perfect DM stack.
\item
$\calx$ is a relatively $i$-perfect algebraic stack (resp. DM stack) if and only if $\caly$ is a relatively $i$-perfect algebraic stack (resp. DM stack).
\item
$\calx$ is a relatively $i$-quasiperfect algebraic stack (resp. DM stack) if and only if $\caly$ is a relatively $i$-quasiperfect algebraic stack (resp. DM stack).
\item
$\calx$ is a relatively $i$-semiperfect algebraic stack (resp. DM stack) if and only if $\caly$ is a relatively $i$-semiperfect algebraic stack (resp. DM stack).
\item
$\calx$ is a relatively $i$-strongly perfect algebraic stack (resp. DM stack) if and only if $\caly$ is a relatively $i$-strongly algebraic stack (resp. DM stack).
\end{enumerate}
\end{lemma}
\begin{proof}
By \cite[Tag03YQ]{Stack Project}, $\calx$ is a algebraic stack (resp. DM stack) if and only if $\caly$ is a algebraic stack (resp. DM stack). Let $x:\SchU\rightarrow\calx$ be a surjective smooth (resp. \'{e}tale) $1$-morphism for $U\in\ObSchU$ a perfect scheme. Choose an equivalence $f:\calx\rightarrow\caly$. Consider the following $2$-commutative diagram
$$
\xymatrix{
  \SchU \ar[d]_{x} \ar[rr]^{id} &  & \SchU \ar[d]^{f\circ x} \\
  \calx \ar[rr]^{f} &  &  \caly  }
$$
Then \cite[Tag03YQ]{Stack Project} shows that $f\circ x$ is surjective smooth (resp. \'{e}tale) if and only if $x$ is surjective smooth (resp. \'{e}tale). This proves (1).

(2), (3), (4), and (5) follow directly from Lemma \ref{L4}. The equivalence $f:\calx\rightarrow\caly$ yields a $2$-commutative diagram
$$
\xymatrix{
  \calx \ar[d]_{\Delta_{\calx}} \ar[rr]^{f} &  & \caly \ar[d]^{\Delta_{\caly}} \\
  \calx\times\calx \ar[rr]^{f\times f} &  & \caly\times\caly   }
$$
whose horizontal arrows are equivalences. Then (6), (7), (8), and (9) all follow directly from Lemma \ref{L5}.
\end{proof}

There will be a series of propositions specifying the 2-fibre products of perfect algebraic stacks (resp. DM stacks). We will show that all 2-categories in Definition \ref{D2} have 2-fibre products.
\begin{proposition}
Let $\calx$ and $\caly$ be algebraic stacks (resp. DM stacks) over $S$.
\begin{enumerate}
\item
Let $\calz$ be a stack in groupoids over $\SchS$ whose diagonal is perfect.  Let $f:\calx\rightarrow\calz$ and $g:\caly\rightarrow\calz$ be $1$-morphisms of stacks in groupoids. If $\calx$ and $\caly$ are perfect, then the $2$-fibre product $\calx\times_{\calz}\caly$ is a perfect algebraic stack (resp. DM stack).
\item
Let $\calz$ be an algebraic stack over $S$. Let $f:\calx\rightarrow\calz$ and $g:\caly\rightarrow\calz$ be $1$-morphisms of algebraic stacks. If $\calx$ and $\caly$ are perfect, then the $2$-fibre product $\calx\times_{\calz}\caly$ is a perfect algebraic stack (resp. DM stack). In particular, if $\calz$ is a perfect algebraic stack (resp. DM stack), then $\calx\times_{\calz}\caly$ is a perfect algebraic stack (resp. DM stack) such that $\calx\times_{\calz}\caly$ is a $2$-fibre product in the $2$-category ${\rm{Perf}}AStack_{S}$ (resp. ${\rm{Perf}}DM_{S}$).
\end{enumerate}
\end{proposition}
\begin{proof}
It follows from \cite[Tag04TF]{Stack Project} that $\calx\times_{\calz}\caly$ is an algebraic stack. Let $U,V\in\ObSchS$ be perfect schemes, and let $x:\SchU\rightarrow\calx,y:\SchV\rightarrow\caly$ be surjective smooth (resp. \'{e}tale) morphisms. Consider the following solid 2-pullback diagram
$$
\xymatrix{
  \SchU\times_{\calz}\SchV \ar[d]_{}\ar@{-->}[dr]_{} \ar[rr]^{} &  & \SchV \ar[d]^{y} \\
  \caly\times_{\calz}\SchU \ar[d]_{} \ar[r]^{} & \calx\times_{\calz}\caly \ar[d]_{} \ar[r]^{} & \caly \ar[d]^{g} \\
  \SchU \ar[r]^{x} & \calx \ar[r]^{f} & \calz   }
$$
The dotted 1-morphism $\SchU\times_{\calz}\SchV\rightarrow\calx\times_{\calz}\caly$ which is the composition of base changes of $x$ and $y$ is smooth (resp. \'{e}tale).

(1): If the diagonal of $\calz$ is perfect, it follows from Proposition \ref{P3} that $\SchU\times_{\calz}\SchV$ is represented by a perfect algebraic space $F$. Let $W\in\ObSchS$ be a perfect scheme such that $h_{W}\rightarrow F$ is surjective \'{e}tale. Then $(Sch/W)_{fppf}\rightarrow\cals_{F}$ is surjective \'{e}tale. Hence, the composition $(Sch/W)_{fppf}\rightarrow\calx\times_{\calz}\caly$ is surjective smooth (resp. \'{e}tale) such that $\calx\times_{\calz}\caly$ is a perfect algebraic stack (resp. DM stack). This proves (1).

(2): For the second statement, note that ${\rm{Perf}}AStack_{S}$ (resp. ${\rm{Perf}}DM_{S}$) is a full sub 2-category of the 2-category $AStack_{S}$ of algebraic stacks over $S$. Thus, if $\calz$ is also a perfect algebraic stack (resp. DM stack), then $\calx\times_{\calz}\caly$ is a 2-fibre product in ${\rm{Perf}}AStack_{S}$ (resp. ${\rm{Perf}}DM_{S}$).
\end{proof}

\begin{proposition}
Let $f:\calx\rightarrow\calz$ and $g:\caly\rightarrow\calz$ be $1$-morphisms of DM stacks.
\begin{enumerate}
\item
If $\calx,\caly,\calz$ are representably weakly perfect, then the $2$-fibre product $\calx\times_{\calz}\caly$ is a representably weakly perfect DM stack. It is also a $2$-fibre product is the $2$-category ${\rm{\widehat{WPerf}}}DM_{S}$.
\item
If $\calx,\caly,\calz$ are representably perfect, then the $2$-fibre product $\calx\times_{\calz}\caly$ is a representably perfect DM stack. It is also a $2$-fibre product is the $2$-category ${\rm{\widehat{Perf}}}DM_{S}$.
\item
If $\calx,\caly,\calz$ are representably quasi-perfect, then the $2$-fibre product $\calx\times_{\calz}\caly$ is a representably quasi-perfect DM stack. It is also a $2$-fibre product is the $2$-category ${\rm{\widehat{QPerf}}}DM_{S}$.
\item
If $\calx,\caly,\calz$ are representably semiperfect, then the $2$-fibre product $\calx\times_{\calz}\caly$ is a representably semiperfect DM stack. It is also a $2$-fibre product is the $2$-category ${\rm{\widehat{SPerf}}}DM_{S}$.
\item
If $\calx,\caly,\calz$ are representably strongly perfect, then the $2$-fibre product $\calx\times_{\calz}\caly$ is a representably strongly perfect DM stack. It is also a $2$-fibre product is the $2$-category ${\rm{\widehat{StPerf}}}DM_{S}$.
\end{enumerate}
\end{proposition}
\begin{proof}
These are clear by Proposition \ref{P4}.
\end{proof}

\begin{proposition}\label{P6}
Let $\calz$ be a stack in groupoids over $\SchS$ whose diagonal is representable by algebraic spaces. Let $\calx$ and $\caly$ be algebraic stacks (resp. DM stacks) over $S$. Let $f:\calx\rightarrow\calz$ and $g:\caly\rightarrow\calz$ be $1$-morphisms of stacks in groupoids.
\begin{enumerate}
  \item
  If $\calx$ and $\caly$ are relatively $i$-perfect, then the $2$-fibre product $\calx\times_{\calz}\caly$ is a relatively $i$-perfect algebraic stack (resp. DM stack).
  \item
  If $\calx$ and $\caly$ are relatively $i$-quasiperfect, then the $2$-fibre product $\calx\times_{\calz}\caly$ is a relatively $i$-quasiperfect algebraic stack (resp. DM stack).
  \item
  If $\calx$ and $\caly$ are relatively $i$-semiperfect, then the $2$-fibre product $\calx\times_{\calz}\caly$ is a relatively $i$-semiperfect algebraic stack (resp. DM stack).
  \item
  If $\calx$ and $\caly$ are relatively $i$-strongly perfect, then the $2$-fibre product $\calx\times_{\calz}\caly$ is a relatively $i$-strongly perfect algebraic stack (resp. DM stack).
\end{enumerate}
\end{proposition}
\begin{proof}
(1): Without loss of generality, take $i=0$. First, it is easy to see that $\calx\times_{\calz}\caly$ is a algebraic stack (resp. DM stack), see the proof of \cite[Tag04TF]{Stack Project}. Let $U\in\ObSchS$. Let $u,v$ be objects in $(\calx\times_{\calz}\caly)_{U}$. Then one can write $u=(x,y,\alpha)$ and $v=(x',y',\alpha')$ where $\alpha:f(x)\rightarrow g(y)$ is an isomorphism, similar for $\alpha'$. It is clear that the diagram
$$
\xymatrix{
  \textit{Isom}(u,v) \ar[d]_{} \ar[rr]^{} &  & \textit{Isom}(y,y') \ar[d]^{} \\
  \textit{Isom}(x,x') \ar[rr]^{} &  & \textit{Isom}(f(x),g(y'))  }
$$
is Cartesian. By Proposition \ref{P1}, $\textit{Isom}(y,y'),\textit{Isom}(x,x'),\textit{Isom}(f(x),g(y'))$ are perfect algebraic space. Thus, $\textit{Isom}(u,v)$ is also a perfect algebraic space such that $\calx\times_{\calz}\caly$ is relatively perfect.

The proof of (2), (3), (4) is similar to (1).
\end{proof}

\begin{proposition}
Let $f:\calx\rightarrow\calz$ and $g:\caly\rightarrow\calz$ be $1$-morphisms of algebraic stacks (resp. DM stacks) over $S$.
\begin{enumerate}
  \item
  If $\calx,\caly,\calz$ are relatively $i$-perfect, then the $2$-fibre product $\calx\times_{\calz}\caly$ is a relatively $i$-perfect algebraic stack (resp. DM stack). It is also a $2$-fibre product is the $2$-category ${\rm{\underline{Perf}}}AStack^{i}_{S}$ (resp. ${\rm{\underline{Perf}}}DM^{i}_{S}$).
  \item
  If $\calx,\caly,\calz$ are relatively $i$-quasiperfect, then the $2$-fibre product $\calx\times_{\calz}\caly$ is a relatively $i$-quasiperfect algebraic stack (resp. DM stack). It is also a $2$-fibre product is the $2$-category $\underline{\mathcal{Q}{\rm{Perf}}}AStack^{i}_{S}$ (resp. $\underline{\mathcal{Q}{\rm{Perf}}}DM^{i}_{S}$).
  \item
  If $\calx,\caly,\calz$ are relatively $i$-semiperfect, then the $2$-fibre product $\calx\times_{\calz}\caly$ is a relatively $i$-semiperfect algebraic stack (resp. DM stack). It is also a $2$-fibre product is the $2$-category $\underline{\mathcal{S}{\rm{Perf}}}AStack^{i}_{S}$ (resp. $\underline{\mathcal{S}{\rm{Perf}}}DM^{i}_{S}$).
  \item
  If $\calx,\caly,\calz$ are relatively $i$-strongly perfect, then the $2$-fibre product $\calx\times_{\calz}\caly$ is a relatively $i$-strongly perfect algebraic stack (resp. DM stack). It is also a $2$-fibre product is the $2$-category $\underline{\mathcal{ST}{\rm{Perf}}}AStack^{i}_{S}$ (resp. $\underline{\mathcal{ST}{\rm{Perf}}}DM^{i}_{S}$).
\end{enumerate}
\end{proposition}
\begin{proof}
We just prove (1) as the other are the same. First, it follows from the stronger Proposition \ref{P6} that the 2-fibre product $\calx\times_{\calz}\caly$ is a relatively $i$-perfect algebraic stack (resp. DM stack). Then the fact that $\calx\times_{\calz}\caly$ is a 2-fibre product in the $2$-category ${\rm{\underline{Perf}}}AStack^{i}_{S}$ (resp. ${\rm{\underline{Perf}}}DM^{i}_{S}$) follows formally from the fact that the $2$-category ${\rm{\underline{Perf}}}AStack^{i}_{S}$ (resp. ${\rm{\underline{Perf}}}DM^{i}_{S}$) is a full sub 2-category of the 2-category $AStack_{S}$ of algebraic stacks over $S$.
\end{proof}

\section{Algebraic Frobenius morphisms}\label{B3}
In this section, we extend the notion of algebraic Frobenius morphisms of algebraic spaces to the setting of algebraic stacks. This enables us to describe a perfect algebraic stack $\calx$ in terms of the endomorphism $\calx\rightarrow\calx$.

As the usual case, the Frobenius morphism of an algebraic stack $\calx$ only makes sense when $\calx$ has characteristic $p$. So we first need to formalize the characteristic of a given category fibred in groupoids whose diagonal is representable by algebraic spaces.

\begin{definition}
Let $\calx$ be a category fibred in groupoids over $\SchS$ whose diagonal is representable by algebraic spaces.
\begin{enumerate}
  \item
  We say that $\calx$ has characteristic $p$ if $\calx$ is nonempty and there exists a surjective smooth $1$-morphism $\SchU\rightarrow\calx$ where $U\in\ObSchS$ has characteristic $p$.
  \item
  $\calx$ is said to have characteristic $0$ if $\calx$ is nonempty and for every surjective smooth $1$-morphism $\SchU\rightarrow\calx$ with $U\in\ObSchS$, $U$ is not an $\mathbb{F}_{p}$-scheme.
\end{enumerate}
We will use ${\rm{char}}(\calx)$ to indicate the characteristic of $\calx$.
\end{definition}

The following lemma ensures that the characteristic of a category fibred in groupoids is independent of the choice of smooth atlases.
\begin{lemma}\label{LL4}
Let $\calx$ be a category fibred in groupoids over $\SchS$ whose diagonal is representable by algebraic spaces. Suppose that there are two surjective smooth $1$-morphisms $\SchU\rightarrow\calx$ and $\SchV\rightarrow\caly$ where $U,V\in\ObSchS$ have characteristics $p,q$. Then we have $p=q$.
\end{lemma}
\begin{proof}
Consider the 2-fibre product $\SchU\times_{\calx}\SchV\cong\cals_{F}$ for some algebraic space $F$ over $S$. Choose a surjective \'{e}tale map $W\rightarrow F$ for some scheme $W\in\ObSchS$. Then the composition $\SchW\rightarrow\SchU\times_{\calx}\SchV\rightarrow\SchU$ gives rise to a unique morphism of schemes $W\rightarrow U$ such that $W$ is an $\mathbb{F}_{p}$-scheme. Similarly, another morphism of schemes $W\rightarrow V$ makes $W$ an $\mathbb{F}_{q}$-scheme. This shows that $p=q$.
\end{proof}

One observes the following lemma concerning morphisms of algebraic stacks in different characteristics.
\begin{lemma}
Let $\calx\rightarrow\caly$ be a $1$-morphism of categories fibred in groupoids over $\SchS$ whose diagonals are representable by algebraic spaces. Suppose that $\calx,\caly$ have characteristics $p,q$. Then we have $p=q$.
\end{lemma}
\begin{proof}
Let $\SchV\rightarrow\caly$ be a surjective smooth 1-morphism where $V\in\ObSchS$ has characteristic $q$. It follows from \cite[Tag04T1]{Stack Project} that there exist $U\in\ObSchS$ and a 2-commutative diagram
$$
\xymatrix{
  \SchU \ar[d]_{} \ar[r]^{} & \SchV \ar[d]^{} \\
  \calx \ar[r]^{} & \caly   }
$$
where $\SchU\rightarrow\calx$ is surjective smooth. This shows that $U$ is an $\mathbb{F}_{q}$-scheme. Now, choose a surjective smooth 1-morphism $\SchW\rightarrow\calx$ where $W\in\ObSchS$ is an $\mathbb{F}_{p}$-scheme. Then by Lemma \ref{LL4}, we have $p=q$.
\end{proof}

The following proposition shows that one can pass between the characteristic of algebraic spaces and categories fibred in groupoids.
\begin{proposition}
Let $\calx$ be a category fibred in groupoids over $\SchS$ whose diagonal is representable by algebraic spaces. Let $F$ be an algebraic space over $S$. Suppose that $\calx$ is representable by $F$. Then $\calx$ has characteristic $p$ if and only if $F$ has characteristic $p$.
\end{proposition}
\begin{proof}
Let $U\in\ObSchS$ in characteristic $p$. Assume that $F$ has characteristic $p$. Then by assumption, there is a surjective \'{e}tale $1$-morphism $\SchU\rightarrow\cals_{F}$. Thus, the composition $\SchU\rightarrow\calx$ is surjective \'{e}tale such that $\calx$ has characteristic $p$. Conversely, if $\calx$ has characteristic $p$, then the composition $\SchU\rightarrow\calx\rightarrow\cals_{F}$ is surjective smooth. Thus, $F$ has characteristic $p$.
\end{proof}

Next, we want to use the similar method as in \cite[Lemma 4.5]{Liang} to define the algebraic Frobenius of algebraic stacks. However, the resulted canonical $1$-morphism does not share all our desired properties. Still, such a canonical $1$-morphism has some other useful properties. Here is the lemma.
\begin{lemma}\label{L8}
Let $\calc$ be a category. Let $p_{0}:\calx\rightarrow\calc,p_{1}:\calx'\rightarrow\calc,p_{2}:\caly\rightarrow\calc,p_{3}:\caly'\rightarrow\calc$ be categories over $\calc$. Let $a:\calx\rightarrow\caly,b:\calx'\rightarrow\caly',c:\calx\rightarrow\calx'$ be $1$-morphisms. Then there exists a $1$-morphism $d:\caly\rightarrow\caly'$ such that the following diagram
$$
\xymatrix{
  \calx \ar[d]_{c} \ar[r]^{a} & \caly \ar@{-->}[d]^{d} \\
  \calx' \ar[r]^{b} & \caly'   }
$$
commutes. Moreover, if $a,b,c$ are equivalences, then $d$ is an equivalence.
\end{lemma}
\begin{proof}
Let $a:\Ob(\calx)\rightarrow\Ob(\caly),c:\Ob(\calx)\rightarrow\Ob(\calx'),b:\Ob(\calx')\rightarrow\Ob(\caly')$ be object-functions. Then the object-function $d:\Ob(\caly)\rightarrow\Ob(\caly')$ is given by
$$
\begin{cases}
d(a(x))=b(c(x)), & \textrm{for all} \ x\in\Ob(\calx);\\
d(y)=x_{0},  & \textrm{for all} \ y\in\Ob(\caly)\backslash\textrm{Im}(a)\textrm{ and some }x_{0}\in\Ob(\caly').
\end{cases}
$$
It is clear that $d$ is well-defined.

Fix $x,y\in\Ob(\calx)$. Let $a:\textrm{Hom}_{\calx}(x,y)\rightarrow\textrm{Hom}_{\caly}(ax,ay),c:\textrm{Hom}_{\calx}(x,y)\rightarrow\textrm{Hom}_{\calx'}(cx,cy), b:\textrm{Hom}_{\calx'}(x,y)\rightarrow\textrm{Hom}_{\caly'}(bx,by)$ be arrow-functions. Then the arrow-function $d:\textrm{Hom}_{\caly}(x,y)\rightarrow\textrm{Hom}_{\caly'}(dx,dy)$ is given by
$$
\begin{cases}
d(a(x'))=b(c(x')), & \textrm{for all} \ x'\in\textrm{Hom}_{\calx}(x,y);\\
d(y')=x_{0}',  & \textrm{for all} \ y'\in\textrm{Hom}_{\caly}(x,y)\backslash\textrm{Im}(a)\textrm{ and some }x_{0}'\in\textrm{Hom}_{\caly'}(x,y).
\end{cases}
$$
These give us the desired functor $d$. And we can show that $d$ is over $\calc$.
\end{proof}

The lemma applies to the following two statements.
\begin{proposition}\label{P8}
Let $\calx$ be a category fibred in groupoids over $\SchS$ which is representable by an algebraic space $F$ over $S$. Suppose that $F$ has characteristic $p$ with algebraic Frobenius morphism $\Psi_{F}:F\rightarrow F$. Then there exists a $1$-morphism $\psi_{\calx}:\calx\rightarrow\calx$ such that the diagram
$$
\xymatrix{
  \cals_{F} \ar[d]_{\cals_{\Psi_{F}}} \ar[r]^{\simeq} & \calx \ar@{-->}[d]^{\psi_{\calx}} \\
  \cals_{F} \ar[r]^{\simeq} & \calx   }
$$
commutes.
\end{proposition}
\begin{proof}
This follows directly from Lemma \ref{L8}.
\end{proof}

\begin{proposition}\label{P7}
Let $\calx$ be a category fibred in groupoids with representable diagonal in characteristic $p$ over $S$. Let $\SchU\rightarrow\calx$ be a surjective smooth $1$-morphism for $U\in\ObSchS$ in characteristic $p$. Then there exists a $1$-morphism $\Psi_{\calx}:\calx\rightarrow\calx$ that fits into the commutative dotted diagram
$$
\xymatrix{
  \SchU \ar[d]_{\cals_{\Phi_{U}}} \ar[rr]^{}& & \calx \ar@{-->}[d]^{\Psi_{\calx}} \\
  \SchU \ar[rr]^{}& & \calx   }
$$
where $\Phi_{U}$ denotes the absolute Frobenius morphism of $U$.
\end{proposition}
\begin{proof}
This follows directly from Lemma \ref{L8}.
\end{proof}

If we speak of a category fibred in groupoids representable by an algebraic space, then the induced 1-morphisms in Proposition \ref{P7} and Proposition \ref{P8} are equivalent.
\begin{proposition}\label{P9}
Let $\calx$ be a category fibred in groupoids in characteristic $p$ over $\SchS$. Suppose that $\calx$ is representable by an algebraic space $F$ over $S$. Then there is a one-to-one correspondence between the $1$-morphisms $\Psi_{\calx}$ in Proposition \ref{P7} and the 1-morphisms $\psi_{\calx}$ in Proposition \ref{P8}.
\end{proposition}
\begin{proof}
Here is the commutative diagram in Proposition \ref{P7}.
$$
\xymatrix{
  \SchU \ar[d]_{\cals_{\Phi_{U}}} \ar[rr]^{}& & \calx \ar[d]^{\Psi_{\calx}} \\
  \SchU \ar[rr]^{}& & \calx   }
$$
This gives rise to a commutative diagram
$$
\xymatrix{
  U \ar[d]_{\cals_{\Phi_{U}}} \ar[rr]^{}& & F \ar[d]^{\Psi_{F}} \\
  U \ar[rr]^{}& & F   }
$$
where $U\rightarrow F$ is surjective \'{e}tale. Thus, $\Psi_{F}$ is the algebraic Frobenius of $F$ and $\Psi_{\calx}$ is the $1$-morphism that makes the diagram
$$
\xymatrix{
  \cals_{F} \ar[d]_{\cals_{\Psi_{F}}} \ar[rr]^{}& & \calx \ar[d]^{\Psi_{\calx}} \\
  \cals_{F} \ar[rr]^{}& & \calx   }
$$
commute, i.e. $\Psi_{\calx}=\psi_{\calx}$.

Now, consider the commutative diagram in Proposition \ref{P8}
$$
\xymatrix{
  \cals_{F} \ar[d]_{\cals_{\Psi_{F}}} \ar[r]^{\simeq} & \calx \ar[d]^{\psi_{\calx}} \\
  \cals_{F} \ar[r]^{\simeq} & \calx   }
$$
This gives rise to a commutative diagram
$$
\xymatrix{
  \SchU \ar[d]_{\cals_{\Phi_{U}}} \ar[rr]^{}& & \calx \ar[d]^{\Psi_{\calx}} \\
  \SchU \ar[rr]^{}& & \calx   }
$$
Thus, we have $\Psi_{\calx}=\psi_{\calx}$.
\end{proof}

Now, we make the definition of canonical morphisms of algebraic stacks.
\begin{definition}
Let $\calx$ be an algebraic stack of characteristic $p$ over $S$. The canonical morphism of $\calx$ is one of the induced $1$-morphisms $\Psi_{\calx}^{*}:\calx\rightarrow\calx$ as in Proposition \ref{P7}.
\end{definition}

The following statement provides us with an alternative description of a perfect algebraic stack that is representable by an algebraic space.
\begin{proposition}\label{T1}
Let $\calx$ be an algebraic stack of characteristic $p$ over $S$ with canonical morphism $\Psi_{\calx}^{*}:\calx\rightarrow\calx$. If $\calx$ is representably perfect, then $\Psi_{\calx}^{*}$ is an equivalence. Conversely, if $\Psi_{\calx}^{*}$ is an equivalence and $\calx$ is representable by an algebraic space over $S$, then $\calx$ is representably perfect.
\end{proposition}
\begin{proof}
The sufficiency follows from Lemma \ref{L8}. Conversely, if $\Psi_{\calx}^{*}$ is an equivalence, then it can be shown that $\cals_{\Psi_{F}}$ is an equivalence. Thus, $\Psi_{F}$ is an isomorphism such that $F$ is perfect. This shows that $\calx$ is representably perfect.
\end{proof}

The following statement provides us with an alternative description of a perfect algebraic stack that is representable by a presheaf of sets.
\begin{proposition}
Let $\calx$ be an algebraic stack of characteristic $p$ over $S$. Suppose that there is an equivalence $\calx\cong\cals_{F}$ for some presheaf of sets $F$ on $\SchS$. Then $\calx$ is perfect if and only if $\Psi_{\calx}^{*}:\calx\rightarrow\calx$ is an equivalence.
\end{proposition}
\begin{proof}
Assume that $\calx$ is perfect. Let $\SchU\rightarrow\calx$ be a surjective smooth $1$-morphism for $U\in\ObSchS$ of characteristic $p$. Then we have a commutative diagram
$$
\xymatrix{
  h_{U} \ar[d]_{h(\Phi_{U})} \ar[rr]^{f} &  & F \ar[d]^{} \\
  h_{U} \ar[rr]^{f} &  & F   }
$$
where $f$ is surjective smooth and $\Phi_{U}$ is the absolute Frobenius of $U$. Next, it follows from \cite[Tag0BGR]{Stack Project} that $F$ is an algebraic space. Let $\{p_{i}\}_{i\in I}$ be a conservative family of points. Now, consider the following commutative diagram of stalks
$$
\xymatrix{
  h_{U,p_{i}} \ar[d]_{h(\Phi_{U})_{p_{i}}} \ar[rr]^{f_{p_{i}}} &  & F_{p_{i}} \ar[d]^{} \\
  h_{U,p_{i}} \ar[rr]^{f_{p_{i}}} &  & F_{p_{i}}   }
$$
where $p_{i}\in\{p_{i}\}_{i\in I}$. One can show that $F_{p_{i}}\rightarrow F_{p_{i}}$ is bijective. Thus, $F\rightarrow F$ is an isomorphism such that $\Psi_{\calx}$ is an equivalence. Conversely, if $\Psi_{\calx}$ is an equivalence, then $F\rightarrow F$ is an isomorphism such that $F_{p_{i}}\rightarrow F_{p_{i}}$ is bijective. This shows that $h(\Phi_{U})_{p_{i}}:h_{U,p_{i}}\rightarrow h_{U,p_{i}}$ is bijective. Hence, $h(\Phi_{U})$ is an isomorphism such that $U$ is perfect. This implies that $\calx$ is perfect.
\end{proof}

The above theorem immediately gives rise to the following proposition.
\begin{proposition}
Let $\calx$ be an algebraic stack of characteristic $p$ over $S$. If $\calx$ is perfect, then $\Psi_{\calx}^{*}:\calx\rightarrow\calx$ is $2$-perfect (resp. $2$-quasiperfect, resp. $2$-semiperfect, resp. $2$-strongly perfect).
\end{proposition}
\begin{proof}
If $\Psi_{\calx}^{*}$ is an equivalence, then we have $\SchU\times_{x,\calx}\calx\cong\SchU$ for perfect scheme $U\in\ObSchS$ and $x\in\Ob(\calx_{U})$. Thus, $\Psi_{\calx}^{*}$ is 2-perfect (resp. $2$-quasiperfect, resp. $2$-semiperfect, resp. $2$-strongly perfect).
\end{proof}

The canonical morphism of an algebraic stack is representable by algebraic spaces.
\begin{proposition}
Let $\calx$ be an algebraic stack in characteristic $p$ over $S$ with canonical morphism $\Psi_{\calx}^{*}:\calx\rightarrow\calx$. Then $\Psi_{\calx}^{*}$ is representable by algebraic spaces.
\end{proposition}
\begin{proof}
If $\calx$ is an empty algebraic stack, then the canonical morphism $\Psi_{\calx}^{*}:\calx\rightarrow\calx$ is trivially representable. Now, assume that $\calx\neq\emptyset$. Let $U,V\in\ObSchS$ be nonempty schemes. Choose an equivalence $\cals_{F}\rightarrow\SchU\times_{\calx}\SchV$ for some algebraic space $F$ over $S$. Consider the composition $\SchU\times_{\calx}\calx\rightarrow\SchU\rightarrow\cals_{F}$. Then we have an equivalence
$$
(\SchU\times_{\calx}\calx)\times_{\cals_{F}}(\SchU\times_{\calx}\SchV)\cong\SchU\times_{\calx}\calx.
$$
Choose $W\simeq U\times_{F}U$ for some $W\in\ObSchS$. Then there are equivalences
\begin{align*}
&(\SchU\times_{\calx}\calx)\times_{\cals_{F}}(\SchU\times_{\calx}\SchV) \\
&\cong(\calx\times_{\calx}\cals_{U\times_{F}U})\times_{\calx}\SchV \\
&\cong\cals_{U\times_{F}U}\times_{\calx}\SchV \\
&\cong\SchW\times_{\calx}\SchV \\
&\cong\cals_{G}
\end{align*}
for some algebraic space $G$ over $S$. Thus, we have $\SchU\times_{\calx}\calx\cong\cals_{G}$.
\end{proof}

The canonical morphism of an algebraic stack is necessarily surjective.
\begin{lemma}\label{P10}
Let $\calx$ be an algebraic stack of characteristic $p$ over $S$ with canonical morphism $\Psi_{\calx}^{*}:\calx\rightarrow\calx$. Then $\Psi_{\calx}^{*}$ is surjective.
\end{lemma}
\begin{proof}
Let $U\in\ObSchS$ of characteristic $p$ together with a surjective smooth $1$-morphism $\SchU\rightarrow\calx$. Consider the commutative diagram
$$
\xymatrix{
  \left|\SchU\right| \ar[d]_{} \ar[rr]^{}& & \left|\calx\right| \ar[d]^{\left|\Psi_{\calx}^{*}\right|} \\
  \left|\SchU\right| \ar[rr]^{}& & \left|\calx\right|   }
$$
of points of algebraic stacks, where the horizontal arrows and the left vertical arrow are surjective. Thus, $\left|\Psi_{\calx}^{*}\right|$ is surjective such that $\Psi_{\calx}^{*}$ is surjective.
\end{proof}

The canonical morphism of a perfect algebraic stack induces a homeomorphism of the underlying topological space.
\begin{lemma}
Let $\calx$ be an algebraic stack of characteristic $p$ over $S$. If $\calx$ is perfect, then $\left|\Psi_{\calx}^{*}\right|:\left|\calx\right|\rightarrow\left|\calx\right|$ is a homeomorphism of the underlying topological space.
\end{lemma}
\begin{proof}
Let $U\in\ObSchS$ be a perfect scheme together with a surjective smooth $1$-morphism $\SchU\rightarrow\calx$. First, it follows from Proposition \ref{P10} that $\left|\Psi_{\calx}^{*}\right|$ is surjective. So one just need to show that $\left|\Psi_{\calx}^{*}\right|$ is injective. Assume that $\calx$ is perfect. Then $f:\left|\SchU\right|\rightarrow\left|\SchU\right|$ is a homeomorphism with an inverse $f^{-1}:\left|\SchU\right|\rightarrow\left|\SchU\right|$. This induces the following commutative diagram
$$
\xymatrix{
  \left|\SchU\right| \ar[d]_{f} \ar[r]^{} & \left|\calx\right| \ar[d]^{\left|\Psi_{\calx}^{*}\right|} \\
  \left|\SchU\right| \ar[d]_{f^{-1}} \ar[r]^{} & \left|\calx\right| \ar[d]^{} \\
  \left|\SchU\right| \ar[r]^{} & \left|\calx\right|   }
$$
which implies that $\left|\Psi_{\calx}^{*}\right|$ has a left inverse such that $\left|\Psi_{\calx}^{*}\right|$ is injective. Thus, $\left|\Psi_{\calx}^{*}\right|$ is a homeomorphism.
\end{proof}

The following lemma characterizes the property of canonical morphism.
\begin{lemma}
Let $\calx$ be an algebraic stack of characteristic $p$ over $S$ with canonical morphism $\Psi_{\calx}^{*}:\calx\rightarrow\calx$. If $\calx$ is representably perfect, then $\Psi_{\calx}^{*}$ is $2$-weakly perfect and representably $2$-perfect.
\end{lemma}
\begin{proof}
It follows from Theorem \ref{T1} that $\Psi_{\calx}^{*}$ is an equivalence. Thus, $\Psi_{\calx}^{*}$ is $2$-weakly perfect and representably $2$-perfect by Lemma \ref{L9}.
\end{proof}

Suppose that we are given an algebraic stack $\calx$ in characteristic $p$ over $S$. Let $f:\SchU\rightarrow\calx$ be a surjective smooth $1$-morphism for $U\in\ObSchS$ in characteristic $p$. Consider the 2-fibre product $\cal{R}=\SchU\times_{f,\calx,f}\SchU$. Choose an equivalence $\cals_{F}\cong\cal{R}$ for some algebraic space $F$ in characteristic $p$ over $U$. Let $s,t:F\rightarrow h_{U}$ be morphisms corresponding to the projections $pr_{0},pr_{1}:\cal{R}\rightarrow\SchU$, which are surjective smooth. Let $c:F\times_{s,h_{U},t}F\rightarrow F$ be the morphism corresponding to the projection $pr_{02}:\cal{R}\times_{pr_{0},\SchU,pr_{1}}\cal{R}\rightarrow\cal{R}$. The quintuple $(h_{U},F,s,t,c)$ forms a smooth groupoid in algebraic spaces over $S$. Moreover, the 1-morphism $f$ gives rise to an equivalence $\calx\cong[h_{U}/F]$ of stacks in groupoids over $\SchS$. Similarly, any DM stack $\calx$ in characteristic $p$ over $S$ gives rise to an \'{e}tale groupoid in algebraic spaces $(h_{U},F,s,t,c)$ over $S$ together with an equivalence $\calx\cong[h_{U}/F]$.

\begin{lemma}
Consider the smooth groupoid in algebraic spaces $(h_{U},F,s,t,c)$ as above. The map $\Psi_{f}:(h_{U},F,s,t,c)\rightarrow(h_{U},F,s,t,c)$ given by the Frobenius $\Phi_{U}:h_{U}\rightarrow h_{U}$ and the algebraic Frobenius $\Psi_{F}:F\rightarrow F$ is an endomorphism of the groupoid in algebraic spaces $(h_{U},F,s,t,c)$.
\end{lemma}
\begin{proof}
Suppose that $\varphi:h_{V}\rightarrow F$ is a surjective \'{e}tale map for $V\in\ObSchS$ in characteristic $p$ and that $e:h_{U}\rightarrow F$ is the identity of the groupoid in algebraic spaces $(h_{U},F,s,t,c)$. Then the composition $h_{V}\xrightarrow{\varphi}F\xrightarrow{s}h_{U}$ comes from a unique morphism of schemes $V\rightarrow U$. In other words, the morphism $s:F\rightarrow h_{U}$ is induced by the morphism of schemes $V\rightarrow U$. Thus, it follows from \cite[Proposition 4.10]{Liang} that the diagram
$$
\xymatrix{
  F \ar[d]_{s} \ar[r]^{\Psi_{F}} & F \ar[d]^{s} \\
  h_{U} \ar[r]^{\Phi_{U}} & h_{U}   }
$$
commutes. Similarly, we can show that $t\circ\Psi_{F}=\Phi_{U}\circ t$ and that the diagram
$$
\xymatrix{
  h_{U} \ar[d]_{\Phi_{U}} \ar[r]^{e} & F \ar[d]^{\Psi_{F}} \\
  h_{U} \ar[r]^{e} & F   }
$$
commutes. Note that there is another commutative diagram
$$
\xymatrix{
  h_{U} \ar[d]_{e} \ar[r]^{e} & F \ar[d]^{s} \\
  h_{U} \ar[r]^{t} & F   }
$$
which gives rise to a unique morphism $h_{U}\rightarrow F\times_{s,h_{U},t}F$. Thus, the second diagram can be factored into the following commutative diagram
$$
\xymatrix{
  h_{U} \ar[d]_{\Phi_{U}} \ar[r]^{} & F\times_{s,h_{U},t}F \ar[d]^{(\Psi_{F},\Psi_{F})} \ar[r]^{\ \ \ \ c} & F \ar[d]^{\Psi_{F}} \\
  h_{U} \ar[r]^{} & F\times_{s,h_{U},t}F \ar[r]^{\ \ \ \ c} & F   }
$$
This shows that the diagram
$$
\xymatrix{
  F\times_{s,h_{U},t}F \ar[d]_{(\Psi_{F},\Psi_{F})} \ar[rr]^{c} & & F \ar[d]^{\Psi_{F}} \\
  F\times_{s,h_{U},t}F \ar[rr]^{c}& & F   }
$$
commutes.
\end{proof}
The morphism $\Psi_{f}:(h_{U},F,s,t,c)\rightarrow(h_{U},F,s,t,c)$ is called the \textit{algebraic Frobenius morphism} of $(h_{U},F,s,t,c)$. It induces a canonical $1$-morphism $[\Psi_{f}]:[h_{U}/F]\rightarrow[h_{U}/F]$ of the quotient stack $[h_{U}/F]$. Thus, we obtain a canonical $1$-morphism $\Psi_{\calx}:\calx\rightarrow\calx$ of the algebraic stack $\calx$.
\begin{definition}
Let $\calx$ be an algebraic stack of characteristic $p$ over $S$. The algebraic Frobenius morphism of $\calx$ is the canonical $1$-morphism $\Psi_{\calx}:\calx\rightarrow\calx$ as above.
\end{definition}

The following proposition shows that perfect algebraic stacks is the same as relatively $1$-perfect algebraic stacks.
\begin{proposition}\label{P11}
Let $\calx$ be an algebraic stack in characteristic $p$ over $S$. Then $\calx$ is perfect if and only if $\calx$ is relatively $1$-perfect.
\end{proposition}
\begin{proof}
Consider the smooth groupoid in algebraic spaces $(h_{U},F,s,t,c)$ together with an equivalence $[h_{U}/F]\cong\calx$ as above. Let $h_{V}\rightarrow F$ be a surjective \'{e}tale map for $V\in\ObSchS$ in characteristic $p$. Then we have a commutative diagram
$$
\xymatrix{
  h_{V} \ar[d]_{\Phi_{V}} \ar[r]^{} & F \ar[d]^{\Psi_{F}} \ar[r]^{s} & h_{U} \ar[d]^{\Phi_{U}} \\
  h_{V} \ar[r]^{} & F \ar[r]^{s} & h_{U}   }
$$
Thus, it follows from \cite[Lemma 5.10]{Liang} that $\Phi_{U}$ is an isomorphism if and only if $\Psi_{F}$ is an isomorphism. This shows that the category fibred in groupoid $\SchU\times_{f,\calx,f}\SchU$ over $\SchU$ is representable by a perfect algebraic space if and only if $\calx$ is perfect.
\end{proof}

This induces an equivalence between the 2-category of perfect algebraic stacks and the 2-category of relatively $1$-perfect algebraic stacks.
\begin{theorem}\label{TT2}
The $2$-category of perfect algebraic stacks over $S$ is equivalent to the $2$-category of relatively $1$-perfect algebraic stacks over $S$, i.e. ${\rm{\underline{Perf}}}AStack^{1}_{S}\cong{\rm{Perf}}AStack_{S}$.
\end{theorem}
\begin{proof}
The statement follows from Proposition \ref{P11} and definitions.
\end{proof}

We arrive at the following theorem which provides us with an equivalent definition of perfect algebraic stacks.
\begin{theorem}\label{T3}
Let $\calx$ be an algebraic stack in characteristic $p$ over $S$ with algebraic Frobenius $\Psi_{\calx}:\calx\rightarrow\calx$. Then $\calx$ is perfect if and only if $\Psi_{\calx}$ is an equivalence.
\end{theorem}
\begin{proof}
Consider the smooth groupoid in algebraic spaces $(h_{U},F,s,t,c)$ together with an equivalence $[h_{U}/F]\cong\calx$ as above. Suppose that $\calx$ is perfect. Then it follows from Proposition \ref{P11} that $\Psi_{f}$ is an isomorphism such that $\Psi_{\calx}$ is an equivalence. Conversely, if $\Psi_{\calx}$ is an equivalence, then it follows from \cite[Tag046T]{Stack Project} that $(h_{U},F,s,t,c)$ is the restriction of $(h_{U},F,s,t,c)$ via the Frobenius $\Phi_{U}:h_{U}\rightarrow h_{U}$. In other words, for each $T\in\ObSchS$, the functor of a groupoid $(h_{U}(T),F(T),s,t,c)\rightarrow(h_{U}(T),F(T),s,t,c)$ is fully faithful, which implies that the algebraic Frobenius $\Psi_{F}:F\rightarrow F$ is an isomorphism. By Proposition \ref{P11}, $\Phi_{U}$ is an isomorphism if and only if $\Psi_{F}$ is an isomorphism. Thus, $\calx$ is perfect.
\end{proof}

We readily deduce the following corollary in terms of the above results.
\begin{corollary}
Let $\calx$ be an algebraic stack in characteristic $p$ over $S$ with algebraic Frobenius $\Psi_{\calx}:\calx\rightarrow\calx$. The following are equivalent:
\begin{enumerate}
  \item
$\calx$ is perfect;
  \item
$\Psi_{\calx}:\calx\rightarrow\calx$ is an equivalence;
  \item
$\calx$ is relatively $1$-perfect.
\end{enumerate}
\end{corollary}
\begin{proof}
Combine Proposition \ref{P11} with Theorem \ref{T3}.
\end{proof}

\section{Perfection of algebraic stacks}\label{B4}
In this section, we focus on the perfection of algebraic stacks, which extends our theory of perfection of algebraic spaces in \cite{Liang1}. First, we observe the following lemma on the perfection of categories fibred in groupoids which are representable by algebraic spaces.

\begin{lemma}\label{L11}
Let $F$ be an algebraic space in characteristic $p$ over $S$. Then we have
$$
\lim_{\substack{\longleftarrow \\ n\in\N}}\cals_{F}=\cals_{\lim_{n\geq0}F}=\cals_{F^{pf}},
$$
where the transition maps are algebraic Frobenius.
\end{lemma}
\begin{proof}
This follows from the equivalence of categories between the category of presheaves of sets over $\SchS$ and the category of categories fibred in sets over $\SchS$.
\end{proof}

Here is the definition of perfections of categories fibred in groupoids representable by an algebraic space.
\begin{definition}
Let $\calx$ be a category fibred in groupoids in characteristic $p$ over $\SchS$. Suppose that there is an equivalence $\cals_{F}\cong\calx$ for some algebraic space $F$ in characteristic $p$ over $S$. Then the perfection of $\calx$ is any category fibred in groupoids $\calx^{pf}$ over $\SchS$ with an equivalence $\calx^{pf}\cong\cals_{F^{pf}}$.
\end{definition}
\begin{remark}
Let $F^{pf}\rightarrow F$ be the canonical projection of $F^{pf}$. This induces a canonical morphism $p_{\calx}:\calx^{pf}\rightarrow\calx$ which is called the \textit{canonical projection} of $\calx^{pf}$.
\end{remark}

Let $\calx$ be an algebraic stack in characteristic $p$ over $S$. Let $U\in\ObSchS$ of characteristic $p$ with surjective smooth $1$-morphism $f:\SchU\rightarrow\calx$. Let $F$ be an algebraic space representing the $2$-fibre product $\SchU\times_{f,\calx,f}\SchU$. Recall that in \S\ref{B4}, every algebraic stack $\calx$ gives rise to a smooth groupoid $(h_{U},F,s,t,c)$ in algebraic spaces such that there is an equivalence $[h_{U}/F]\cong\calx$, where $s,t:F\rightarrow h_{U}$ are surjective smooth morphisms and $c:F\times_{s,h_{U},t}F\rightarrow F$ denotes the projection to the second $F$. The following theorem ensures that the perfection of arbitrary algebraic stack in characteristic $p$ exists.
\begin{theorem}\label{T2}
Let $\calx$ be an algebraic stack (resp. a DM stack) in characteristic $p$ over $S$. Let $U\in\ObSchS$ be a scheme in characteristic $p$. Let $(h_{U},F,s,t,c)$ be a smooth (resp. an \'{e}tale) groupoid in algebraic spaces over $S$ as above such that there is an equivalence $\calx\cong[h_{U}/F]$. Then the quotient stack $[h_{U}^{pf}/F^{pf}]$ is a perfect algebraic stack (resp. DM stack).
\end{theorem}
\begin{proof}
Suppose that $\calx$ is an algebraic stack in characteristic $p$ over $S$. First, we see that the quotient stack $[h_{U}^{pf}/F^{pf}]$ is a stack in groupoids by construction. By \cite[Lemma 7.2]{Liang1}, the perfection $(h_{U}^{pf},F^{pf},s^{\natural},t^{\natural},c^{\natural})$ is still a groupoid in algebraic spaces over $S$. Thus, the diagonal of $[h_{U}^{pf}/F^{pf}]$ is representable by algebraic spaces, see \cite[Tag04WZ]{Stack Project}.

Let $T\in\ObSchS$ and let $x:(Sch/T)_{fppf}\rightarrow[h_{U}^{pf}/F^{pf}]$ be a $1$-morphism. Next, we just need to show that the projection
$$
(Sch/T)_{fppf}\times_{[h_{U}^{pf}/F^{pf}]}(Sch/U^{pf})_{fppf}\longrightarrow(Sch/T)_{fppf}
$$
is surjective smooth. Assume that $x$ comes from $x:T\rightarrow h_{U}^{pf}$. Then it follows from the proof of \cite[Tag04X0]{Stack Project} that we have an equality
$$
(Sch/T)_{fppf}\times_{[h_{U}^{pf}/F^{pf}]}(Sch/U^{pf})_{fppf}=\cals_{F^{pf}}\times_{(Sch/U^{pf})_{fppf}}(Sch/T)_{fppf}.
$$
Now, consider the Cartesian diagram as follows.
$$
\xymatrix{
  F^{pf}\times_{s^{\natural},h_{U}^{pf},x}T \ar[d]_{} \ar[rr]^{}& & T \ar[d]^{x} \\
  F\times_{h_{U}}h_{U}^{pf}\simeq F^{pf} \ar[d]_{} \ar[rr]^{\ \ s^{\natural}}& & h_{U}^{pf} \ar[d]^{} \\
  F \ar[rr]^{s} & & h_{U}   }
$$
The projection $F^{pf}\times_{s^{\natural},h_{U}^{pf},x}T\rightarrow T$ is smooth as the base change of the smooth morphism $s:F\rightarrow h_{U}$ of algebraic spaces. And it is surjective since $s^{\natural}$ admits a section. Therefore, the $1$-morphism $x:(Sch/U^{pf})_{fppf}\rightarrow[h_{U}^{pf}/F^{pf}]$ is surjective smooth. This proves that the quotient stack $[h_{U}^{pf}/F^{pf}]$ is a perfect algebraic stack.

For the second statement, note that $(h_{U}^{pf},F^{pf},s^{\natural},t^{\natural},c^{\natural})$ is an \'{e}tale groupoid in algebraic spaces by \cite[Proposition 7.9]{Liang1}. Thus, it follows from \cite[Tag04TK]{Stack Project} that the quotient stack $[h_{U}^{pf}/F^{pf}]$ is an algebraic stack. And one can show that the $1$-morphism $(Sch/U^{pf})_{fppf}\rightarrow[h_{U}^{pf}/F^{pf}]$ is surjective \'{e}tale. This proves that the quotient stack $[h_{U}^{pf}/F^{pf}]$ is a perfect DM stack.
\end{proof}
Consider the canonical projection $p:(h_{U}^{pf},F^{pf},s^{\natural},t^{\natural},c^{\natural})\rightarrow(h_{U},F,s,t,c)$. This induces a canonical $1$-morphism of quotient stacks $[p]:[h_{U}^{pf}/F^{pf}]\rightarrow[h_{U}/F]$. Now, we can make the definition of perfections of algebraic stacks (resp. DM stacks).
\begin{definition}
Consider the situation as in Theorem \ref{T2}. The perfection of $\calx$ is any perfect algebraic stack (resp. DM stack) $\calx^{pf}$ over $S$ such that there is an equivalence $\calx^{pf}\cong[h_{U}^{pf}/F^{pf}]$. The canonical $1$-morphism $p_{\calx}:\calx^{pf}\rightarrow\calx$ induced by $[p]$ above is called the \textit{canonical projection} of $\calx^{pf}$.
\end{definition}

The following proposition shows that the perfection of an algebraic stack (resp. a DM stack) in characteristic $p$ can be described as an inverse limit.
\begin{proposition}\label{P12}
Consider the situation as in Theorem \ref{T2}. Let $\Psi_{f}$ be the algebraic Frobenius of $(h_{U},F,s,t,c)$ and let $\Psi_{\calx}$ be the algebraic Frobenius of $\calx$. Then there is an isomorphism
$$
[h_{U}^{pf}/F^{pf}]\cong\lim_{\substack{\longleftarrow \\ n\in\N}}[h_{U}/F].
$$
where the transition maps are algebraic Frobenius $[\Psi_{f}]:[h_{U}/F]\rightarrow[h_{U}/F]$. Moreover, we have the following string of equivalences
$$
\calx^{pf}\cong[h_{U}^{pf}/F^{pf}]\cong\lim_{\substack{\longleftarrow \\ n\in\N}}[h_{U}/F]\cong\lim_{\substack{\longleftarrow \\ n\in\N}}\calx,
$$
where the transition maps of the second limit are algebraic Frobenius $\Psi_{\calx}:\calx\rightarrow\calx$.
\end{proposition}
\begin{proof}
Consider the category fibred in groupoids $[h_{U}^{pf}/_{p}F^{pf}]=[\lim_{n\in\N}h_{U}/_{p}\lim_{n\in\N}F]$ which corresponds to the presheaf of groupoids
\begin{align*}
\SchS^{opp}&\longrightarrow Groupoids \\
S'&\longmapsto(\lim_{\substack{\longleftarrow \\ n\in\N}}h_{U}(S'),\lim_{\substack{\longleftarrow \\ n\in\N}}F(S'),s^{\natural},t^{\natural},c^{\natural}).
\end{align*}
We claim that $[\lim_{n\in\N}h_{U}/_{p}\lim_{n\in\N}F]=\lim_{n\in\N}[h_{U}/_{p}F]$. First, it is easy to check that there is an isomorphism
$$
\Ob([\lim_{n\in\N}h_{U}/_{p}\lim_{n\in\N}F])=\Ob(\lim_{n\in\N}[h_{U}/_{p}F]).
$$
Moreover, for $x,y\in\Ob([\lim_{n\in\N}h_{U}/_{p}\lim_{n\in\N}F])$, there is an isomorphism
$$
\textrm{Hom}_{[\lim_{n\in\N}h_{U}/_{p}\lim_{n\in\N}F]}(x,y)=\textrm{Hom}_{\lim_{n\in\N}[h_{U}/_{p}F]}(x,y).
$$
Thus, we obtain an isomorphism of categories fibred in groupoids
$$
\lim_{\substack{\longleftarrow \\ n\in\N}}[h_{U}/_{p}F]=[\lim_{\substack{\longleftarrow \\ n\in\N}}h_{U}/_{p}\lim_{\substack{\longleftarrow \\ n\in\N}}F].
$$
By \cite[Proposition 2.1.9]{Talpo}, limits commute with stackification. This shows that there is an isomorphism of quotient stacks
$$
\lim_{\substack{\longleftarrow \\ n\in\N}}[h_{U}/F]=[\lim_{\substack{\longleftarrow \\ n\in\N}}h_{U}/\lim_{\substack{\longleftarrow \\ n\in\N}}F].
$$
\end{proof}

The perfection of an algebraic stack (resp. a DM stack) in characteristic $p$ satisfies the following universal property.
\begin{corollary}\label{C1}
Let $\calx$ be an algebraic stack (resp. a DM stack) in characteristic $p$ over $S$ with perfection $\calx^{pf}$. Suppose that we are given a perfect algebraic stack $\caly$ over $S$ together with a $1$-morphism $f:\caly\rightarrow\calx$. Then there exists a unique $1$-morphism $f^{pf}:\caly\rightarrow\calx^{pf}$ such that the following diagram
$$
\xymatrix{
  \calx^{pf}  \ar[dr]_{p_{\calx}}
                &  &    \caly \ar@{-->}[ll]_{f^{pf}} \ar[dl]^{f}    \\
                &  \calx                }
$$
is commutative.
\end{corollary}
\begin{proof}
This is a consequence of Proposition \ref{P12} and \cite[Proposition 2.1.4]{Talpo}.
\end{proof}

We record here a lemma that will be useful in the sequel.
\begin{lemma}\label{L12}
Let $\calx$ be an algebraic stack (resp. a DM stack) of characteristic $p$ over $S$ with perfection $\calx^{pf}$. Then there is an isomorphism $(\calx^{pf})^{pf}\cong \calx^{pf}$.
\end{lemma}
\begin{proof}
It can be shown using universal properties of $\calx^{pf}$ and $(\calx^{pf})^{pf}$ that $(\calx^{pf})^{pf}$ satisfies the same universal property as $\calx^{pf}$.
\end{proof}

Next, we prove the functoriality of perfection of algebraic stacks (resp. DM stacks).
\begin{lemma}\label{L10}
Let $f:\calx\rightarrow\caly$ be a $1$-morphism of algebraic stacks (resp. DM stacks) in characteristic $p$ over $S$. Then $f$ induces a canonical $1$-morphism $f^{\natural}:\calx^{pf}\rightarrow\caly^{pf}$ of perfect algebraic stacks (resp. DM stacks) over $S$ such that the diagram
$$
\xymatrix{
  \calx^{pf} \ar@{-->}[d]_{f^{\natural}} \ar[r]^{p_{\calx}} & \calx \ar[d]^{f} \\
  \caly^{pf} \ar[r]^{p_{\caly}} & \caly   }
$$
is commutative.
\end{lemma}
\begin{proof}
This follows directly from the universal property described in Corollary \ref{C1} that there exists a unique $1$-morphism $f^{\natural}=(f\circ p_{\calx})^{pf}$ such that $p_{\caly}\circ f^{\natural}=f\circ p_{\calx}$.
\end{proof}

Let $AStack^{p}_{S}$ (resp. $DM^{p}_{S}$) be the $2$-category of algebraic stacks (resp. DM stacks) in characteristic $p$ over $S$. It is a sub 2-category of the 2-category $AStack^{p}_{S}$ (resp. $DM^{p}_{S}$) defined as follows:
\begin{enumerate}
  \item
Its objects will be algebraic stacks (resp. DM stacks) $\calx$ over $S$ with ${\rm{char}}(\calx)=p$.
  \item
Its $1$-morphisms will be functors of categories over $\SchS$.
  \item
Its $2$-morphisms will be transformations between functors over $\SchS$.
\end{enumerate}

Then there is a natural $2$-functor
$$
\underline{{\rm{Perf}}}_{S}^{\mathbb{A}}:AStack^{p}_{S}\longrightarrow{\rm{Perf}}AStack_{S}, \ \ \calx\longmapsto\calx^{pf}
$$
that to each $1$-morphism $f:\calx\rightarrow\caly$, it assigns the canonical $1$-morphism $f^{\natural}:\calx^{pf}\rightarrow\caly^{pf}$ as in Lemma \ref{L10}. And to each $2$-morphism in $AStack^{p}_{S}$, it associates some $2$-morphism in ${\rm{Perf}}AStack_{S}$. Such a 2-functor is called the \textit{algebraic perfection $2$-functor}. When we speak of the \textit{algebraic perfection functor}, we mean the ordinary functor between the underlying categories induced by the algebraic perfection $2$-functor.

Similarly, there is another natural $2$-functor
$$
\underline{{\rm{Perf}}}_{S}^{\mathbb{D}}:DM^{p}_{S}\longrightarrow{\rm{Perf}}DM_{S}, \ \ \calx\longmapsto\calx^{pf}.
$$
Such a 2-functor is called the \textit{DM perfection $2$-functor}. When we speak of the \textit{DM perfection functor}, we mean the ordinary functor between the underlying categories induced by the DM perfection $2$-functor.

Let $i:{\rm{Perf}}DM_{S}\rightarrow DM^{p}_{S}$ be the inclusion functor. The following proposition characterizes the property of the DM perfection functor.
\begin{proposition}\label{P13}
The DM perfection functor ${\rm{\underline{Perf}}}_{S}^{\mathbb{D}}$ is right adjoint to the inclusion functor $i$. In other words, for any $\calx\in{\rm{Ob}}({\rm{Perf}}DM_{S})$ and $\caly\in{\rm{Ob}}(DM^{p}_{S})$, there exists a functorial bijection
$$
{\rm{Hom}}_{DM^{p}_{S}}(\calx,\caly)\longrightarrow{\rm{Hom}}_{{\rm{Perf}}DM_{S}}(\calx,\caly^{pf}),\ \ f\longmapsto f^{pf}.
$$
\end{proposition}
\begin{proof}
The inverse function is given by
$$
{\rm{Hom}}_{{\rm{Perf}}DM_{S}}(\calx,\caly^{pf})\longrightarrow{\rm{Hom}}_{DM^{p}_{S}}(\calx,\caly), \ \ g\longmapsto g\circ p_{\caly},
$$
where $\calx\in{\rm{Ob}}({\rm{Perf}}DM_{S})$ and $\caly\in{\rm{Ob}}(DM^{p}_{S})$.

Next, we will show that the bijection defined above is functorial. Let $B'\xrightarrow{h}B$ be a morphism in ${\rm{Perf}}DM_{S}$. Let $B\xrightarrow{f}A\xrightarrow{g}A'$ be morphisms in $DM^{p}_{S}$. Let $A^{pf},A'^{pf}$ be the perfections of $A,A'$. It follows from Lemma \ref{L10} that $gp_{A}=p_{A'}g^{\natural}$, which implies that $gp_{A}f^{pf}h=p_{A'}g^{\natural}f^{pf}h$. Then the universal property of the perfection $A^{pf}$ yields $g\circ f\circ h=p_{A'}\circ g^{\natural}\circ f^{pf}\circ h$. Finally, by the universal property of $A'^{pf}$, we have $p_{A'}\circ (g\circ f\circ h)=p_{A'}\circ (g^{\natural}\circ f^{pf}\circ h)$. The uniqueness requirement yields $g\circ f\circ h=g^{\natural}\circ f^{pf}\circ h$. Thus, the bijection is natural in both $X$ and $Y$.
\end{proof}

Let $i':{\rm{Perf}}AStack_{S}\rightarrow AStack^{p}_{S}$ be the inclusion functor. The following proposition characterizes the property of the algebraic perfection functor.
\begin{proposition}
The algebraic perfection functor ${\rm{\underline{Perf}}}_{S}^{\mathbb{A}}$ is right adjoint to the inclusion functor $i'$. In other words, for any $\calx\in{\rm{Ob}}({\rm{Perf}}AStack_{S})$ and $\caly\in{\rm{Ob}}(AStack^{p}_{S})$, there exists a functorial bijection
$$
{\rm{Hom}}_{AStack^{p}_{S}}(\calx,\caly)\longrightarrow{\rm{Hom}}_{{\rm{Perf}}AStack_{S}}(\calx,\caly^{pf}),\ \ f\longmapsto f^{pf}.
$$
\end{proposition}
\begin{proof}
The proof is similar to Proposition \ref{P13}.
\end{proof}
In practice, ${\rm{\underline{Perf}}}_{S}^{\mathbb{A}}$ shares almost the same properties with ${\rm{\underline{Perf}}}_{S}^{\mathbb{D}}$. Therefore, in the following, we will not distinguish between the algebraic perfection $2$-functor ${\rm{\underline{Perf}}}_{S}^{\mathbb{A}}$ and the DM perfection $2$-functor ${\rm{\underline{Perf}}}_{S}^{\mathbb{D}}$. We will simply call ${\rm{\underline{Perf}}}_{S}^{\mathbb{A}}$ the \textit{perfection $2$-functor} and denote it by ${\rm{\underline{Perf}}}_{S}$. When we speak of the \textit{perfection functor}, we mean the ordinary functor induced by the perfection $2$-functor. Moreover, we will only state the results on algebraic stacks though they also hold for DM stacks.

The following lemma shows that the perfection functor ${\rm{\underline{Perf}}}_{S}$ is full but not faithful.
\begin{lemma}\label{L15}
Let $\calx,\caly$ be algebraic stacks in characteristic $p$ over $S$ with perfections $\calx^{pf},\caly^{pf}$. Let $f:\calx^{pf}\rightarrow\caly^{pf}$ be a $1$-morphism of perfect algebraic stacks over $S$. Then there exists a canonical $1$-morphism $f^{-1}:\calx\rightarrow\caly$ such that the diagram
$$
\xymatrix{
  \calx^{pf} \ar[d]_{f} \ar[r]^{p_{\calx}} & \calx \ar@{-->}[d]^{f^{-1}} \\
  \caly^{pf} \ar[r]^{p_{\caly}} & \caly   }
$$
is commutative.
\end{lemma}
\begin{proof}
It follows from Lemma \ref{L8} that there exists a $1$-morphism $f^{-1}$ that makes the diagram commute.
\end{proof}

In the following, we will study the nature of the perfection functor. We can show that the perfection $2$-functor commutes with 2-fibre products and products in the 2-category $AStack^{p}$.
\begin{proposition}\label{P14}
Let $\calx,\caly,\calz$ be algebraic stacks in characteristic $p$ over $S$. Let $f:\calx\rightarrow\calz$ and $g:\caly\rightarrow\calz$ be $1$-morphisms of algebraic stacks over $S$. Then we have $(\calx\times_{f,\calz,g}\caly)^{pf}=\calx^{pf}\times_{f^{\natural},\calz^{pf},g^{\natural}}\caly^{pf}$. In particular, we have $(\calx\times_{\SchS}\caly)^{pf}=\calx^{pf}\times_{(Sch/S^{pf})_{fppf}}\caly^{pf}$.
\end{proposition}
\begin{proof}
It follows from the explicit description of $(\calx\times_{\calz}\caly)^{pf}$ that every element in $(\calx\times_{\calz}\caly)^{pf}$ is of the form
$$
((U,x_{0},y_{0},f_{0}),(U,x_{1},y_{1},f_{1}),...,(U,x_{n},y_{n},f_{n}),...),
$$
where $U\in\ObSchS,x_{i}\in\Ob(\calx_{U}),y_{i}\in\Ob(\caly_{U})$, and $f_{i}:f(x_{i})\rightarrow g(y_{i})$ is an isomorphism in $\calz_{U}$. Similarly, every element in $\calx^{pf}\times_{\calz^{pf}}\caly^{pf}$ can be written as the form $(U,x,y,f')$ where $U\in\ObSchS,x\in\Ob(\calx_{U}^{pf}),y\in\Ob(\caly_{U}^{pf})$, and $f':f^{\natural}(x)\rightarrow g^{\natural}(y)$ is an isomorphism in $\calz_{U}^{pf}$. Note that there is an isomorphism
$$
((U,x_{0},y_{0},f_{0}),...,(U,x_{n},y_{n},f_{n}),...)\longmapsto(U,(x_{0},...,x_{n},...),(y_{0},...,y_{n},...),(f_{0},...,f_{n},...)).
$$
Since $(x_{0},...,x_{n},...)\in\Ob(\calx_{U}^{pf}),(y_{0},...,y_{n},...)\in\Ob(\caly_{U}^{pf})$, and $(f_{0},...,f_{n},...)$ is an isomorphism in $\calz_{U}^{pf}$, we obtain an isomorphism $(\calx\times_{\calz}\caly)^{pf}\cong\calx^{pf}\times_{\calz^{pf}}\caly^{pf}$.
\end{proof}

The perfection $2$-functor maps $1$-morphisms representable by algebraic spaces to $2$-perfect $1$-morphisms.
\begin{proposition}\label{P15}
Let $f:\calx\rightarrow\caly$ be a $1$-morphism of algebraic stacks in characteristic $p$ over $S$ representable by algebraic spaces. Then the canonical $1$-morphism $f^{\natural}:\calx^{pf}\rightarrow\caly^{pf}$ is $2$-perfect.
\end{proposition}
\begin{proof}
Let $y:\SchU\rightarrow\caly$ be a $1$-morphism of categories fibred in groupoids over $\SchS$ where $U\in\ObSchS$ is a perfect scheme in characteristic $p$. Choose an equivalence $\cals_{F}\cong\SchU\times_{y,\caly}\calx$ where $F$ is an algebraic space in characteristic $p$ over $U$. Then it follows from Proposition \ref{P14} and Lemma \ref{L11} that we have an equivalence $\cals_{F^{pf}}\cong(Sch/U)_{fppf}\times_{y^{\natural},\caly^{pf}}\calx^{pf}$. Since the perfection functor is full by Lemma \ref{L15}, $f$ is 2-perfect.
\end{proof}

Moreover, the perfection of an algebraic stack is relatively $2$-perfect.
\begin{proposition}
Let $\calx$ be an algebraic stack in characteristic $p$ over $S$ with perfection $\calx^{pf}$. Then $\calx^{pf}$ is relatively $2$-perfect.
\end{proposition}
\begin{proof}
It follows from Proposition \ref{P15} that the perfection functor maps the diagonal $\Delta:\calx\rightarrow\calx\times\calx$ to a $2$-perfect diagonal $\Delta^{\natural}:\calx^{pf}\rightarrow\calx^{pf}\times\calx^{pf}$. Thus, $\calx^{pf}$ is relatively $2$-perfect.
\end{proof}

Let $0\leq j\leq2$ be an integer. The above proposition shows that every perfect algebraic stack is both relatively $j$-perfect and perfect.
\begin{lemma}\label{LL2}
Let $\calx$ be an algebraic stack in characteristic $p$ over $S$ with algebraic Frobenius $\Psi_{\calx}:\calx\rightarrow\calx$. Suppose that $\calx$ is perfect. Then we have an equivalence $\calx\cong\calx^{pf}$ such that $\calx$ is relatively $j$-perfect.
\end{lemma}
\begin{proof}
By Theorem \ref{T3}, if $\calx$ is perfect, then the algebraic Frobenius $\Psi_{\calx}$ is an equivalence. Thus, the inverse limit $\lim_{\Psi_{\calx}}\calx$ is equivalent to $\calx$. Moreover, the last statement follows from Lemma \ref{L13}.
\end{proof}

The following theorem specifies a string of equivalences of $2$-categories.
\begin{theorem}
The $2$-category of perfect algebraic stacks over $S$ is equivalent to the $2$-category of relatively $2$-perfect algebraic stacks over $S$, and the $2$-category of relatively $2$-perfect algebraic stacks over $S$ is equivalent to the $2$-category of relatively $1$-perfect algebraic stacks over $S$. In other words, we have the following string of equivalences of $2$-categories
$$
{\rm{Perf}}AStack_{S}\cong\underline{{\rm{Perf}}}AStack^{1}_{S}\cong\underline{{\rm{Perf}}}AStack^{2}_{S}.
$$
\end{theorem}
\begin{proof}
Note that we have an equivalence $\underline{{\rm{Perf}}}AStack^{1}_{S}={\rm{Perf}}AStack_{S}$ by Theorem \ref{TT2}. Then the statement follows from Lemma \ref{LL2} and the inclusion $\underline{{\rm{Perf}}}AStack^{2}_{S}\subset\underline{{\rm{Perf}}}AStack^{1}_{S}={\rm{Perf}}AStack_{S}$.
\end{proof}

Now, we have the following strings of inclusion functors
\begin{align}
&{\rm{Perf}}DM_{S}\subset\underline{\textrm{Perf}}DM^{j}_{S}\subset\underline{\mathcal{Q}\textrm{Perf}}DM^{j}_{S}\subset\underline{\mathcal{S}\textrm{Perf}}DM^{j}_{S}\subset\underline{\mathcal{ST}\textrm{Perf}}DM^{j}_{S}, \\
&{\rm{Perf}}AStack_{S}\subset\underline{\textrm{Perf}}AStack^{j}_{S}\subset\underline{\mathcal{Q}\textrm{Perf}}AStack^{j}_{S}\subset\underline{\mathcal{S}\textrm{Perf}}AStack^{j}_{S}\subset\underline{\mathcal{ST}\textrm{Perf}}AStack^{j}_{S}.
\end{align}

\begin{lemma}
Let $\mathcal{V},\calx$ be algebraic stacks of characteristic $p$ over $S$ and let $\mathcal{V}\rightarrow\calx$ be a $1$-morphism of algebraic stacks over $S$. Then we have an isomorphism $\mathcal{V}^{pf}\cong \mathcal{V}\times_{\calx}\calx^{pf}$. Moreover, the natural projection $p_{\calx}:\calx^{pf}\rightarrow \calx$ is a universal homeomorphism.
\end{lemma}
\begin{proof}
It can checked that the perfection $\mathcal{V}^{pf}$ satisfies the universal property of the base change $\mathcal{V}\times_{\calx}\calx^{pf}$. This gives rise to an isomorphism $\mathcal{V}^{pf}\cong \mathcal{V}\times_{\calx}\calx^{pf}$. For the second statement, choose a smooth cover $\SchU\rightarrow\calx$ for $U\in\ObSchS$ of characteristic $p$. Then the base change of $\calx^{pf}\rightarrow\calx$ by $\SchU\rightarrow\calx$ is $\SchU^{pf}\cong\SchU\times_{\calx}\calx^{pf}\rightarrow\SchU$. It follows from \cite[Remark 5.4]{Bertapellea} that the canonical projection $U^{pf}\rightarrow U$ is a universal homeomorphism. Therefore, $\calx^{pf}\rightarrow\calx$ is a universal homeomorphism by \cite[Tag0DTQ]{Stack Project}.
\end{proof}

We observe the following lemma about points of the perfection.
\begin{lemma}\label{LL1}
Let $\calx$ be an algebraic stack in characteristic $p$ over $S$ with perfection $\calx^{pf}$. Let $p_{\calx}:\calx^{pf}\rightarrow\calx$ be the canonical projection. Then $p_{\calx}$ is a monomorphism. Furthermore, we have a underlying homeomorphism of topological spaces
$$
\left|\calx\right|\xrightarrow{\cong}\left|\calx^{pf}\right|.
$$
\end{lemma}
\begin{proof}
Note that the diagonal $\Delta_{p_{\calx}}:\calx^{pf}\rightarrow\calx^{pf}\times_{\calx}\calx^{pf}\cong\calx^{pf}$ is an equivalence. This shows that $p_{\calx}$ is a monomorphism. The second statement is clear since the canonical projection $\calx^{pf}\rightarrow\calx$ is a universal homeomorphism.
\end{proof}

An algebraic stack is representable by an algebraic space if and only if its perfection is representable by an algebraic space.
\begin{lemma}\label{L6}
Let $\calx$ be an algebraic stack of characteristic $p$ over $S$ with perfection $\calx^{pf}$. Then $\calx$ is an algebraic space if and only if $\calx^{pf}$ is an algebraic space. In particular, $\calx$ is a scheme (resp. affine) if and only if $\calx^{pf}$ is a scheme (resp. affine).
\end{lemma}
\begin{proof}
If $\calx$ is an algebraic space, then clearly $\calx^{pf}$ is an algebraic space. Conversely, consider the quotient stack $[h_{U}/F']\cong\calx$ induced by the smooth groupoid in algebraic spaces $(h_{U},F',s,t,c)$. Now, assume that there are equivalences $\cals_{F}\cong\calx^{pf}$ for some algebraic space $F$ over $S$. This gives rise to an equivalence $[h_{U}^{pf}/F'^{pf}]\cong\cals_{F}$. Thus, the quotient stack $[h_{U}/F']$ is an algebraic space by the definition. For the second statement, see \cite[Lemma 4.21]{Liang1}.
\end{proof}

A morphism of algebraic stacks is representable by algebraic spaces if and only if its perfection is representable by algebraic spaces.
\begin{proposition}\label{P16}
Let $f:\calx\rightarrow\caly$ be a morphism of algebraic stacks in characteristic $p$ and let $f^{\natural}:\calx^{pf}\rightarrow\caly^{pf}$ be its perfection. Then $f$ is representable by algebraic spaces if and only if $f^{\natural}$ is representable by algebraic spaces.
\end{proposition}
\begin{proof}
Let $U\in\ObSchS$ of characteristic $p$ with $\xi\in\Ob(\caly_{U})$. Choose $\xi$ to be some composition $\SchU\rightarrow\caly^{pf}\rightarrow\caly$. Then there are equivalences $\SchU\times_{\caly}\calx\cong\SchU\times_{\caly^{pf}}(\calx\times_{\caly}\caly^{pf})\cong\SchU\times_{\caly^{pf}}\calx^{pf}\cong\cals_{F}$ for some algebraic space $F$ over $U$. This shows that $f^{\natural}$ is representable by algebraic spaces.

Conversely, assume that $f^{\natural}$ is representable by algebraic spaces. Choose an equivalence $\cals_{F}\cong\SchU\times_{\caly^{pf}}\calx^{pf}$ for some algebraic space $F$ over $U$ in characteristic $p$ and $\xi'\in\Ob(\caly^{pf}_{U})$. Then there is an equivalence $\cals_{F^{pf}}\cong(Sch/U^{pf})_{fppf}\times_{\caly^{pf}}\calx^{pf}$. Since the perfection functor is full, it follows from Lemma \ref{L6} that $\SchU\times_{\caly}\calx$ is an algebraic space for every $\xi''\in\Ob(\caly_{U})$. Thus, $f$ is representable by algebraic spaces.
\end{proof}

The following statement enables us to pass between the usual world and the perfect world.
\begin{lemma}\label{T4}
Let $f:\calx\rightarrow\caly$ be a $1$-morphism of algebraic stacks in characteristic $p$ over $S$. Suppose that $f$ is representable by algebraic spaces. Let $\cal{P}$ be a property of morphisms of algebraic spaces which
\begin{enumerate}
\item
is preserved under arbitrary base change, and
\item
is fppf local on the base.
\end{enumerate}
If $f$ has property $\cal{P}$, then $f^{\natural}:\calx^{pf}\rightarrow\caly^{pf}$ has property $\cal{P}$.
\end{lemma}
\begin{proof}
By \cite[Tag03YK]{Stack Project}, if $f$ has property $\cal{P}$, then the morphism of algebraic spaces $F\rightarrow U$ induced by $\cals_{F}\cong\calx\times_{\caly}\SchU\rightarrow\SchU$, where $U\in\ObSchS$ and $F$ is an algebraic space over $U$, has property $\cal{P}$.  Now, choose $\xi$ to be some composition $\SchU\rightarrow\caly^{pf}\rightarrow\caly$. Then there are isomorphisms $\SchU\times_{\caly}\calx\cong\SchU\times_{\caly^{pf}}(\calx\times_{\caly}\caly^{pf})\cong \SchU\times_{\caly^{pf}}\calx^{pf}$. Thus, the morphism $\cals_{F}\cong\SchU\times_{\caly^{pf}}\calx^{pf}\rightarrow\SchU$ is the same as $F\rightarrow U$ which has property $\cal{P}$. This shows that the perfection $f^{\natural}:\calx^{pf}\rightarrow\caly^{pf}$ has property $\cal{P}$.
\end{proof}

More generally, we can extend the results of Lemma \ref{T4} to $1$-morphisms of algebraic stacks that are not necessarily representable by algebraic spaces. We first observe the following lemma.
\begin{lemma}\label{LL5}
Let $f:\calx\rightarrow\caly$ be a $1$-morphism of algebraic stacks in characteristic $p$ over $S$. Let $\cal{P}$ be a property of morphisms of algebraic spaces which is smooth local on the source-and-target. Consider commutative diagrams
$$
\xymatrix{
  U \ar[d]_{a} \ar[r]^{h} & V \ar[d]^{b} \\
  \calx \ar[r]^{f} & \caly   }
$$
where $U,V$ are algebraic spaces and $a,b$ are smooth. The following are equivalent:
\begin{enumerate}
  \item
for any diagram as above such that $U\rightarrow\calx\times_{\caly}V$ is smooth, the morphism $h$ of algebraic spaces has property $\cal{P}$.
  \item
for some diagram as above where $a,b$ is surjective and $U,V$ have characteristic $p$, the morphism $h$ of algebraic spaces has property $\cal{P}$.
\end{enumerate}
\end{lemma}
\begin{proof}
$(1)\Rightarrow(2):$ Choose an algebraic space $V$ in characteristic $p$ with a surjective smooth $1$-morphism $V\rightarrow\caly$. Next, choose an algebraic space $U$ with a surjective smooth $1$-morphism $U\rightarrow\calx\times_{\caly}V$. It is easy to check that $U$ has characteristic $p$. Then the composition $U\rightarrow\calx\times_{\caly}V\rightarrow\calx$ is surjective smooth. Hence, we obtain a diagram as in (2).

$(2)\Rightarrow(1):$ This is obvious by \cite[Tag06FM]{Stack Project}.
\end{proof}

Given a property of morphisms of algebraic spaces which is smooth local on the source-and-target, one can use it to define a corresponding property of morphisms of algebraic stacks. Here we specialize the definition in \cite[Tag06FN]{Stack Project} to the following case.
\begin{definition}
Let $f:\calx\rightarrow\caly$ be a morphism of algebraic stacks in characteristic $p$ over $S$ and let $\cal{P}$ be a property of morphisms of algebraic spaces which is smooth local on source-and-target. We say that $f$ has property $\cal{P}$ if one of the equivalent conditions in Lemma \ref{LL5} is satisfied.
\end{definition}

Properties of morphisms of algebraic stacks corresponding to properties of morphisms of algebraic spaces which are smooth local on the source-and-target are preserved under the perfection functor.
\begin{lemma}\label{T5}
Let $f:\calx\rightarrow\caly$ be a $1$-morphism of algebraic stacks in characteristic $p$ over $S$. Let $\cal{P}$ be a property of morphisms of algebraic spaces which
\begin{enumerate}
\item
is smooth local on the source-and-target, and
\item
is preserved under the perfection functor on algebraic spaces.
\end{enumerate}
If $f$ has property $\cal{P}$, then $f^{\natural}:\calx^{pf}\rightarrow\caly^{pf}$ has property $\cal{P}$.
\end{lemma}
\begin{proof}
Consider the commutative diagram as in Lemma \ref{LL5}
$$
\xymatrix{
  U \ar[d]_{a} \ar[r]^{h} & V \ar[d]^{b} \\
  \calx \ar[r]^{f} & \caly   }
$$
where $U,V$ are algebraic spaces in characteristic $p$ over $S$ and $a,b$ are surjective smooth. Suppose that $h$ has property $\cal{P}$ such that $f$ has property $\cal{P}$. This gives rise to another commutative diagram
$$
\xymatrix{
  U^{pf} \ar[d]_{a^{\natural}} \ar[r]^{h^{\natural}} & V^{pf} \ar[d]^{b^{\natural}} \\
  \calx^{pf} \ar[r]^{f^{\natural}} & \caly^{pf}   }
$$
where $a^{\natural},b^{\natural}$ are surjective smooth by Lemma \ref{T4}. Thus, by assumption, $h^{\natural}$ also has property $\cal{P}$ such that $f^{\natural}$ has property $\cal{P}$.
\end{proof}

Given a property of morphisms of algebraic spaces which is \'{e}tale-smooth local on the source-and-target, one can use it to define a corresponding property of DM morphisms of algebraic stacks. Recall the definition imposing properties on DM morphisms in \cite[Tag06FN]{Stack Project}. We can show that such properties of DM morphisms are preserved under the perfection functor.
\begin{lemma}\label{T6}
Let $f:\calx\rightarrow\caly$ be a DM morphism of algebraic stacks in characteristic $p$. Let $\cal{P}$ be a property of morphisms of algebraic spaces which
\begin{enumerate}
\item
is \'{e}tale-smooth local on the source-and-target, and
\item
is preserved under the perfection functor on algebraic spaces.
\end{enumerate}
If $f$ has property $\cal{P}$, then $f^{\natural}:\calx^{pf}\rightarrow\caly^{pf}$ has property $\cal{P}$.
\end{lemma}
\begin{proof}
Let $U,V$ be algebraic spaces over $S$. It follows from Lemma \ref{T4} and Proposition \ref{P14} that if $V\rightarrow\caly$ is smooth and $U\rightarrow\calx\times_{\caly}V$ is \'{e}tale, then $V^{pf}\rightarrow\caly^{pf}$ is smooth and $U^{pf}\rightarrow\calx^{pf}\times_{\caly^{pf}}V^{pf}$ is \'{e}tale. Moreover, note that $f^{\natural}$ is also DM by Proposition \ref{PP1} below. The rest of the proof is similar to Lemma \ref{T5} above.
\end{proof}

In the following proposition, we summarize the properties of $1$-morphisms that can be passed between the perfection functor.
\begin{proposition}\label{PP1}
Let $f:\calx\rightarrow\caly$ be a $1$-morphism of algebraic stacks in characteristic $p$ over $S$ and let $f^{\natural}:\calx^{pf}\rightarrow\caly^{pf}$ be its perfection. Then the following properties hold for $f$ if and only if they hold for $f^{\natural}$:
\begin{enumerate}
\item
surjective,
\item
quasi-compact,
\item
(universally) submersive
\item
(universally) closed,
\item
(universally) open,
\item
a (universal) homeomorphism,
\item
affine,
\item
integral,
\item
quasi-separated.
\end{enumerate}
If the following properties hold for $f$, then they also hold for $f^{\natural}$:
\begin{enumerate}[start=10]
\item
a monomorphism,
\item
a closed immersion,
\item
an open immersion,
\item
an immersion,

\item
DM,
\item
\'{e}tale,
\item
(faithfully) flat.
\end{enumerate}
\end{proposition}
\begin{proof}
(1)-(2) follow directly from Lemma \ref{LL1} and definitions. For (3)-(6), note that the canonical projection $\calx^{pf}\rightarrow\calx$ is a universal homeomorphism.

(7)-(8): The direct direction follows directly from Lemma \ref{T4} since $f$ has to be representable. The inverse direction is similar to the proof of \cite[Proposition 4.28]{Liang1}.

(9)-(13): These follow from \cite[Proposition 4.28]{Liang1} and Lemma \ref{T4}.

(14): If $f$ is DM, then the second diagonal $\Delta_{f,2}:\calx\rightarrow\calx\times_{(\calx\times_{\caly}\calx)}\calx$ is \'{e}tale. Since $\Delta_{f,2}$ is representable, it follows from \cite[Proposition 4.28]{Liang1} that $\Delta_{f,2}^{\natural}:\calx^{pf}\rightarrow\calx^{pf}\times_{(\calx^{pf}\times_{\caly^{pf}}\calx^{pf})}\calx^{pf}$ is \'{e}tale. As the second diagonal $\Delta_{f,2}^{\natural}$ is locally of finite type, this shows that the diagonal $\Delta_{f}^{\natural}$ is unramified, see \cite[Tag0CJ0]{Stack Project}. Hence, $f^{\natural}$ is DM.

Finally, (15) is by (14) and Lemma \ref{T6} above. And (16) is by Lemma \ref{T5} and (1).
\end{proof}

The perfection functor preserves all kinds of substacks of an algebraic stack.
\begin{proposition}
Let $\calx$ be an algebraic stack in characteristic $p$ over $S$.
\begin{enumerate}
  \item
If $\calx'\subset\calx$ is an open substack, then $\calx'^{pf}$ is an open substack of $\calx^{pf}$.
  \item
If $\calx'\subset\calx$ is a closed substack, then $\calx'^{pf}$ is a closed substack of $\calx^{pf}$.
  \item
If $\calx'\subset\calx$ is a locally closed substack, then $\calx'^{pf}$ is a locally closed substack of $\calx^{pf}$.
\end{enumerate}
\end{proposition}
\begin{proof}
First, we show that any substack of $\calx$ also has characteristic $p$. Let $V\in\ObSchS$ in characteristic $p$ and let $\SchV\rightarrow\calx$ be a surjective smooth $1$-morphism. Then it follows from \cite[Tag04T1]{Stack Project} that there exist $U\in\ObSchS$ and a 2-commutative diagram
$$
\xymatrix{
  \SchU \ar[d]_{} \ar[rr]^{} &  & \SchV \ar[d]^{} \\
  \calx' \ar[rr]^{} &  & \calx   }
$$
where $\SchU\rightarrow\calx'$ is surjective smooth. Thus, this shows that $\calx'$ has characteristic $p$. Now, consider the perfection $\calx'^{pf}\subset\calx^{pf}$ of $\calx'$. By the explicit description of $\calx'^{pf}$, one easily see that $\calx'^{pf}$ is a strictly full subcategory of $\calx^{pf}$. Then (1)-(3) follow from Proposition \ref{PP1} (10)-(12).
\end{proof}

The following statement enables us to pass between the properties of an algebraic stack and its perfection.
\begin{lemma}\label{T7}
Let $\calx$ be an algebraic stack in characteristic $p$ over $S$. Let $\cal{P}$ be a property of schemes which
\begin{enumerate}
  \item
which is local in the smooth topology, and
  \item
which is preserved under the perfection functor on schemes.
\end{enumerate}
If $\calx$ has property $\cal{P}$, then $\calx^{pf}$ has property $\cal{P}$.
\end{lemma}
\begin{proof}
Let $U\rightarrow\calx$ be a surjective smooth $1$-morphism, where $U\in\ObSchS$ has characteristic $p$ and has property $\cal{P}$. This gives rise to a canonical surjective smooth $1$-morphism $U^{pf}\rightarrow\calx^{pf}$ such that $\calx^{pf}$ has property $\cal{P}$.
\end{proof}

Here is a consequence of Theorem \ref{T7} above.
\begin{corollary}\label{C2}
Let $\calx$ be an algebraic stack in characteristic $p$ over $S$. If $\calx$ is perfect, then $\calx$ is reduced. In particular, the perfection $\calx^{pf}$ of $\calx$ is reduced, i.e. $\calx^{pf}=\calx_{red}^{pf}$.
\end{corollary}
\begin{proof}
The first statement follows from \cite[Lemma 3.4]{Liang1} and definitions. For the second statement, one just need to apply Theorem \ref{T7}.
\end{proof}

The following proposition shows that the algebraic Frobenius of an algebraic stack shares the same properties as the absolute Frobenius of a scheme.
\begin{proposition}
Let $\calx$ be an algebraic stack in characteristic $p$ over $S$ with algebraic Frobenius morphism $\Psi_{\calx}:\calx\rightarrow\calx$. Then $\Psi_{\calx}$ is representable by algebraic spaces. Moreover, $\Psi_{\calx}$ is surjective, integral, and is a universal homeomorphism.
\end{proposition}
\begin{proof}
Consider the morphism $\Psi_{\calx}^{\natural}:\calx^{pf}\rightarrow\calx^{pf}$, which is an equivalence due to Theorem \ref{T3}. Then it follows from Proposition \ref{P16} that $\Psi_{\calx}$ is representable by algebraic spaces. Moreover, since $\Psi_{\calx}^{\natural}$ is clearly surjective, integral, and universal homeomorphic, $\Psi_{\calx}$ is surjective, integral, and universal homeomorphic by Proposition \ref{PP1}.
\end{proof}

\section{Comparison with Zhu's perfect algebraic stacks}\label{B5}
In this section, we will compare our theory of perfect algebraic stacks with Zhu's perfect algebraic stacks in \cite{Zhu1}. Let $k$ be a perfect field of characteristic $p$. Let $(Sch/k)_{\textit{fpqc}}$ be the big fpqc site with perfection $(Sch/k)_{\textit{fpqc}}^{pf}$. Let $(\textit{Aff}/k)_{\textit{fpqc}}$ be the big affine fpqc site with perfection $(\textit{Aff}/k)^{pf}_{\textit{fpqc}}$.

We first recall the definition of perfectly smooth morphisms in \cite[Definition A.1.9]{Zhu1}.
\begin{definition}
A map $f:X\rightarrow Y$ of algebraic spaces over $S$ is perfectly smooth at $x\in X$ if there are an \'{e}tale atlas $U\rightarrow X$ at $x$ and an \'{e}tale atlas $V\rightarrow Y$ at $f(x)$ such that the composition $U\rightarrow X\rightarrow Y$ factors as $U\xrightarrow{h} V\rightarrow Y$ and $h$ factors as $U\xrightarrow{h'}V\times(\mathbb{A}^{n})^{pf}\xrightarrow{{\rm{pr}}}V$, where $h'$ is \'{e}tale and ${\rm{pr}}$ is the projection.

We say that $f$ is \textit{perfectly smooth} if it is perfectly smooth at every point in $X$.
\end{definition}

We denote by $\textit{Groupoids}$ the 2-category of groupoids. Here is the definition of perfect algebraic stacks given in \cite[Definition A.1.10]{Zhu1}.
\begin{definition}
A perfect algebraic stack in the sense of Zhu over $k$ is a contravariant $2$-functor
$$
X:(\textit{Aff}/k)^{pf}_{\textit{fpqc}}\longrightarrow \textit{Groupoids}
$$
such that
\begin{enumerate}
\item
The presheaf $X$ is a fpqc sheaf.
  \item
The diagonal is represented by a perfect algebraic space in the sense of Zhu over $k$.
  \item
There exists a perfectly smooth surjective map $U\rightarrow X$ from a perfect algebraic space $U$ in the sense of Zhu over $k$.
\end{enumerate}
\end{definition}

It follows from \cite[Theorem 6.6]{Liang1} that an algebraic space over $k$ is perfect in the sense of Zhu if and only if it is perfect. Thus, one can rewrite the definition to the following form that is easier to deal with.
\begin{definition}
A perfect algebraic stack in the sense of Zhu over $k$ is an algebraic stack $\calx$ over $k$ such that the following properties are satisfied:
\begin{enumerate}
\item
For every perfect schemes $U,V$ over $k$ and any $x\in\Ob(X_{U}),y\in\Ob(X_{V})$, the $2$-fibre product $U\times_{x,\calx,y}V$ is a perfect algebraic space over $k$.
  \item
There exists a perfectly smooth surjective map $U\rightarrow\calx$ from a perfect algebraic space $U$ over $k$.
\end{enumerate}
\end{definition}

One easily observes the following lemma.

\begin{lemma}\label{LL3}
Let $X$ be a perfect algebraic stack in the sense of Zhu over $k$. Then the diagonal $\Delta:X\rightarrow X\times X$ is $2$-perfect. Thus, $X$ is relatively $2$-perfect and perfect.
\end{lemma}
\begin{proof}
It is clear by definitions that $X$ is relatively 2-perfect. Therefore, $X$ is also relatively 1-perfect. Then it follows from Proposition \ref{P11} that $X$ is perfect.
\end{proof}

Let $\mathcal{Z}AS_{k}^{pf}$ be the $2$-category of perfect algebraic stacks in the sense of Zhu over $k$. It is defined as follows.
\begin{enumerate}
  \item
Its objects will be perfect algebraic stacks in the sense of Zhu over $k$.
  \item
Its $1$-morphisms will be functors of categories over $(Sch/k)^{pf}_{\textit{fpqc}}$.
  \item
Its $2$-morphisms will be transformations between functors over $(Sch/k)^{pf}_{\textit{fpqc}}$.
\end{enumerate}

Then we have the following strings of inclusion 2-functors
\begin{align}
&\mathcal{Z}AS_{k}^{pf}\subset\underline{\textrm{Perf}}AStack^{2}_{k}\subset\underline{\mathcal{Q}\textrm{Perf}}AStack^{2}_{k}\subset\underline{\mathcal{S}\textrm{Perf}}AStack^{2}_{k}\subset\underline{\mathcal{ST}\textrm{Perf}}AStack^{2}_{k}, \\
&\mathcal{Z}AS_{k}^{pf}\subset\underline{\textrm{Perf}}AStack^{2}_{k}\subset\underline{\textrm{Perf}}AStack^{1}_{k}={\rm{Perf}}AStack_{k}.
\end{align}

The following theorem specifies the equivalence between Zhu's perfect algebraic stacks and our perfect algebraic stacks.
\begin{theorem}
Let $X$ be an algebraic stack over $k$. Then $X$ is perfect in the sense of Zhu if and only if the following statements are satisfied:
\begin{enumerate}
  \item
$X$ is relatively $2$-perfect.
  \item
All associated algebraic spaces of the diagonal $\Delta:X\rightarrow X\times X$ is a perfect algebraic space $U$.
  \item
There exists a surjective perfectly smooth $1$-morphism $U\rightarrow X$.
\end{enumerate}
\end{theorem}
\begin{proof}
This is clear by Lemma \ref{LL3}, \cite[Theorem 6.6]{Liang1}, and definitions.
\end{proof}

\end{document}